\input ./style/arxiv-vmsta.cfg
\documentclass[numbers,compress,v1.0.1]{vmsta}

\usepackage{mathbh}
\usepackage{vtexurl}
\usepackage{enumerate}

\volume{4}
\issue{1}
\pubyear{2017}
\firstpage{25}
\lastpage{63}
\doi{10.15559/17-VMSTA73}

\newtheorem{thm}{Theorem}
\newtheorem{lemma}{Lemma}
\newtheorem{pro}{Proposition}
\newtheorem{cor}{Corollary}

\theoremstyle{definition}
\newtheorem{defin}{Definition}
\newtheorem{remark}{Remark}

\startlocaldefs
\newcommand{\rrvert}{\vert}
\newcommand{\llvert}{\vert}
\def\nn{\nonumber}
\def\cadlag{{c\`adl\`ag}~}
\def\R{\mathbb{R}}
\def\P{\mathbb{P}}
\def\Fc{{\cal F}}
\def\t{\tau}
\def\Sum{\sum}
\def\Sup{\sup}
\def\Lim{\lim}
\def\N{\mathbb{N}}
\def\M{\mathbb{M}}
\def\E{\mathbb{E}}
\def\P{\mathbb{P}}
\def\L{\mathbb{L}}
\def\t{\tau}
\def\1{\mathbh{1}}
\def\cadlag{{c\`adl\`ag}~}
\def\E{\mathbb{E}}
\def\cB{{\cal B}}
\def\cC{{\cal C}}
\def\cF{{\cal F}}
\def\cH{{\cal H}}
\def\cL{{\cal L}}
\def\cM{{\cal M}}
\def\cP{{\cal P}}
\def\cS{{\cal S}}
\def\cT{{\cal T}}
\def\cU{{\cal U}}
\def\cV{{\cal V}}

\endlocaldefs

\begin{document}
\begin{frontmatter}

\title{$\boldsymbol{\L^p(p\ge2)}$-solutions of generalized BSDEs with jumps and monotone generator in a general filtration}
\author[a]{\inits{M.}\fnm{M'hamed}\snm{Eddahbi}}\email{m.eddahbi@uca.ma}
\author[a]{\inits{I.}\fnm{Imade}\snm{Fakhouri}\corref{cor1}}\email{imadefakhouri@gmail.com}
\cortext[cor1]{Corresponding author.}
\author[a,b]{\inits{Y.}\fnm{Youssef}\snm{Ouknine}}\email{ouknine@uca.ac.ma}
\address[a]{Cadi Ayyad University, Av. Abdelkrim Khattabi, 40000,~Gu\'eliz--Marrakesh,~Morocco}
\address[b]{Hassan II Academy of Sciences and Technology, Morocco}

\markboth{M. Eddahbi et al.}{$\L^p(p\ge2)$-solutions of GBSDEs with
jumps in a general filtration}

\begin{abstract}
In this paper, we study multidimensional generalized BSDEs that have a
monotone generator in a general filtration supporting a Brownian motion
and an independent Poisson random measure. First, we prove the
existence and uniqueness of $\L^p(p\ge2)$-solutions in the case of a
fixed terminal time under suitable $p$-integrability conditions on the
data. Then, we extend these results to the case of a random terminal
time. Furthermore, we provide a comparison result in dimension~$1$.
\end{abstract}


\received{28 April 2016}
%
\revised{18 January 2017}
%
\accepted{19 January 2017}
\publishedonline{7 February 2017}
\end{frontmatter}

\section{Introduction}
This paper is concerned with the study of multidimensional generalized
backward stochastic differential equations (GBSDEs) with jumps in a
general filtration. For convenience of the discussion, let us first
make precise the notion of such equations, which is adopted from \cite
{KP14}.

Let $T>0$ be a fixed time horizon and consider a filtered probability
space $(\varOmega,\Fc,(\Fc_t)_{t\le T},\P)$ carrying a standard
$d$-dimensional Brownian motion $W$ and an independent compensated
Poisson random measure $\widehat\pi$. The filtration
$(\Fc_t)_{0\le t\le T}$ is assumed to be complete and right continuous.
Assume that we are given an $\mathbb{R}^k$-valued $\cF_T$-measurable random
variable $\xi$, a random function $f:\varOmega\times[0,T]\times\R
^k\times\R^{k\times d}\times\cL_\lambda^2\longrightarrow\R^k$
(see Section~2 for the definition of $\cL_{\lambda}^2$) such that
$f(\cdot,y,z,v)$ is $(\cF_t)$-progressively measurable for each
$(y,z,v)$, and an $(\cF_t)_{t\ge0}$-progressively measurable \cadlag
finite-variation process $(R_t)_{t\in[0,T]}$ such that $R_0=0$.
Roughly speaking, solving a GBSDE with jumps in a general filtration
with terminal time $T$ associated with terminal condition $\xi$ and
generator $f+dR$ amounts to finding the usual triple $(Y_t,Z_t,V_t)$
(with $Y$ adapted and $Z$ and $V$ predictable) and a \cadlag martingale
$M=(M_t)_{t\in[0,T]}$ that is orthogonal to $W$ and $\widehat\pi$
(see Lemma~\ref{loc-mart-decomposition}) such that the following
equation is satisfied~$\P$-a.s.:
\begin{align}
\label{RBSDE-jumps}
Y_t&=\xi+\int^T_tf(s,Y_s,Z_s,V_s)ds+\int_t^TdR_s -\int^T_tZ_s dW_s\nn\\
&\quad -\int_{t}^{T}\int_UV_{s}(e)\widehat\pi(de,ds)-\int_t^TdM_s, \quad t\in[0,T].
\end{align}
This equation is usually denoted by GBSDE$(\xi,f+dR)$. Note that the
reason behind adding the martingale $M$ to the definition of GBSDE
\eqref{RBSDE-jumps} is the fact that we do not assume that the
underlying filtration is generated by $W$ and~$\widehat\pi$, and in
such cases, the martingale representation property may fail.

Nonlinear BSDEs with jumps (i.e., the underlying filtration is
generated by a Brownian motion and an independent Poisson random
measure) were first introduced by Tang and Li \cite{TL94}. They proved
the existence and uniqueness of the solution under a Lipschitz
continuity condition on the generator w.r.t.\ the variables. Since
then, a lot of papers (see, e.g., \cite{P97,BBP97,Royer06,Yao10,QS13,Liwei14,KP14},
and the references therein) and books (see, e.g., \cite{Situ05} and
\cite{Delong13}) studied BSDEs with jumps due to the connections of
this subject with mathematical finance (see, e.g., \cite{Delong13})
(e.g., if the Brownian motion stands for the noise from a financial
market, then the Poisson random measure can be interpreted as the
randomness of the insurance claims), stochastic control (see, e.g.,
\cite{KP2015}), and partial differential equations (see, e.g., \cite
{BBP97}), and so on. Since the work of Tang and Li \cite{TL94}, the
attempts of generalization of their results have been made in several
different directions. First of all, many papers aimed at relaxing the
Lipschitz condition on the generator w.r.t.\ $y$. For example, Pardoux
\cite{P97} considered a monotonicity condition on the generator w.r.t.
$y$ and a linear growth condition on $y$. Some efforts were devoted to
weaken the square integrability on the coefficients, for example, $ \E
 [|\xi|^{2}+\int_0^T|f(t,0,0,0)|^2dt ]<+\infty$.
Yao \cite{Yao10} gave the existence and uniqueness results for $\L
^p$-solutions ($p>1$) for BSDEs with jumps for a monotonic generator
(not the same monotonicity condition considered in our paper) and $\L
^p$ coefficients. Later, Li and Wei \cite{Liwei14} analyzed fully
coupled BSDEs with jumps and showed the existence and uniqueness of
\hbox{$\L^p$-solutions} ($p\ge2$) for such equations for a monotone
generator and $p$-integrable data.
Further, other settings of BSDEs with jumps have been introduced.
Pardoux \cite{P97} studied a class of BSDEs with jumps called
generalized BSDEs with jumps that involves an integral w.r.t. an
increasing continuous process. The author shows the uniqueness and
existence of the GBSDE with generator monotone in
$y$ and square-integrable data.

Recently, Kruse and Popier \cite{KP14} considered another direction of
generalization concerning the underlying filtration, which is no longer
assumed to be generated by $W$ and $\widehat\pi$. In fact, they
studied multidimensional BSDEs in a~general filtration of type \eqref
{RBSDE-jumps} with $R\equiv0$.
The authors established the existence and uniqueness of $\L
^p$-solutions ($p>1$) 
under a monotonicity assumption on the generator $f$ w.r.t. $y$ 
and under the condition that the data $\xi$
and $f(t, 0, 0, 0)$ are in $\L^p$, $p > 1$, that is,
%
\begin{equation}
\label{Lp-integrability-xi-f} \E \Biggl[|\xi|^{p}+\int_0^T\big|f(t,0,0,0)\big|^p
dt \Biggr]<+\infty.
\end{equation}
Moreover, they also considered the case of a random terminal time that
is not necessarily bounded.

In our paper, we first investigate the existence and uniqueness of $\L
^p(p\ge2)$-solutions (see Definition~\ref
{Lp-solution-GBSDE-definition}) for GBSDEs \eqref{RBSDE-jumps}
in a deterministic time horizon. We suppose that $f$ is monotonic
w.r.t. $y$ (this condition is essential in the study of BSDEs with
random terminal time), Lipschitz continuous w.r.t. to $z$ and $v$, and
satisfies a very general growth condition w.r.t. $y$ considered earlier
in the Brownian setting in \cite{briandetal03} and recently in the
case of jumps in \cite{KP14}:
%
\begin{equation}
\label{H4-introduction} \forall r>0, \quad\Sup_{|y|\le r}\big|f(\cdot,y,0,0)-f(
\cdot,0,0,0)\big| \in\L^1\bigl(\varOmega\times[0,T]\bigr).
\end{equation}
This condition seems to be the best possible growth condition on $f$
w.r.t.~$y$ and is widely used in the theory of partial differential
equations (see \cite{Benilanetal95} and the references therein).

Concerning the data, we assume that a $p$-integrability condition is
satisfied (see assumption $(H1)$). Moreover, under an additional
assumption on the jump component of $f$ (see $(H6')$ in Section~\ref{section-CT}), we provide a comparison principle in dimension one (see
the counterexample in \cite{BBP97}).
Then, we extend the results obtained in the case of deterministic time
horizon to the case of a random terminal time that is not necessarily bounded.

Let us highlight the main contribution of the paper compared to the
existing literature.
On the one hand, our results extend the work of Kruse and Popier \cite
{KP14} to the case of generalized BSDEs. Furthermore, we strengthen
their results even in the case $R=0$ since our $p$-integrability
condition on $f(t,0,0,0)$ (see assumption $(H1)$) is weaker than the
$\L^p$-integrability \eqref{Lp-integrability-xi-f} assumed in their
paper. It should be mentioned that, due to the $p$-integrability
assumed on $f(t,0,0,0)$ and also to the process $dR$, some difficulties
arise. Indeed, as in \cite{briandetal03} and \cite{KP14}, to study
the $\L^p$-solutions, a result on the existence and uniqueness in the
classical $\L^2$ case (see Theorem \ref{existence-uniqueness-in-L2-multi-case}) is first needed. To obtain
such a~result, the main trick is to truncate the coefficients with
suitable truncation functions in order to have a bounded solution $Y$,
which is a key tool in the proof in the $\L^2$ case (see \cite
[Prop.~2.4 and Thm.~2.2]{P99}, Proposition 2.2 in \cite{BrCar2000}
used in \cite[Thm.~4.2]{briandetal03} and \cite[Lemma 4 and Thm.~1]{KP14}).
The approach followed in the papers mentioned to get this important
estimate fails in our context. This is the reason why we give
nonstandard estimates in Lemma~\ref{Lemma9}, which allow us to
overcome this problem.

On the other hand, we generalize the work of Pardoux \cite{P97} to the
situation of a general filtration. Moreover, even in the case of a
Wiener--Poisson filtration (filtration generated by $W$ and $\pi$),
compared to \cite{P97}, we weaken the growth condition on $f$ w.r.t.
$y$ stated in assumption \eqref{H4-introduction}, instead of the
linear growth condition on the variable~$y$, and derive the existence
and uniqueness of $\L^p$-solutions of our GBSDE, whereas only the
classical $\L^2$-solutions were studied in \cite{P97}. Note also
that, in our case, GBSDE involves an integral w.r.t.\ a
finite-variation \cadlag process unlike \cite{P97}, where an integral
w.r.t. a continuous increasing process is considered instead.

Our main motivation for writing this paper is because it is a first
step in the study of our future work on the existence and uniqueness of
$\L^p$-solutions for reflected GBSDEs. Note that since we solve that
problem using a penalization method, the comparison principle obtained
here is primordial. 
Finally, to the best of our knowledge, there is no such result in the
literature.

The rest of the paper is organized as follows: in the following
section, we give the mathematical setting
of this paper and some basic identities. In Section $3$, we study the
existence and uniqueness of $\L^p$-solutions
on a fixed time interval, which is done in three parts. First, we study
the classical case of $\L^2$-solutions. The proof method follows the arguments
and techniques (convolution, weak convergence, truncation technique)
given in \cite{KR13,P97,P99,briandetal03,KP14} with obvious modifications. Then, in the remaining
parts, we extend the result
to $\L^p$-solutions for any $p\ge2$, using the right a priori
estimate, the $\L^2$ case result, and a truncation technique. In
Section $4$, a comparison
principle for GBSDEs with jumps in dimension $1$ is provided. Finally,
Section $5$ is devoted to the case of a random terminal time.

\section{Preliminaries}
Throughout this paper, $T>0$ is a fixed time horizon, $(\varOmega,\Fc
,(\Fc_t)_{t\le T},\P)$ is a filtered probability space.
The filtration $(\Fc_t, 0\le t\le T)$ is assumed to be complete and
right continuous.
We suppose that $(\varOmega,\Fc,(\Fc_t)_{t\le T},\P)$ supports a
$d$-dimensional Wiener process $(W_t, 0\le t\le T)$ and a random
Poisson measure $\pi$ on $\R^+\times U$, where $U:=\R^n\backslash\{
0\}$ is equipped with its Borel field $\cal U$, with the compensator
$\nu(dt,de)=dt\lambda(de)$ such that $\{\widehat{\pi}([0,t]\times
A)=(\pi-\nu)([0,t]\times A)\}_{t\le T}$ is a martingale for all $A\in
\cal U$ satisfying $\lambda(A)<+\infty$. Here, $\lambda$ is assumed
to be a $\sigma$-finite L\'evy measure on $(U,\cal U)$ such that
\[
\int_{U}\bigl(1\wedge|e|^2\bigr)\lambda(de)<
\infty.
\]

Let \textbf{$\cP$} denote the $\sigma$-algebra of predictable sets
on $\varOmega\times[0,T]$, and let us introduce the following notation:
\begin{itemize}
\item$\cS$ is the set of all adapted \cadlag processes.
\item$G_{\mathit{loc}}(\pi)$ is the set of $\cP\times\cU$-measurable
functions $V$ on $\varOmega\times[0,T] \times U$ such that, for any $t
\geq0$,
\[
\int_0^t \int_U
\bigl(\big|V_s(e)\big|^2\wedge\big|V_s(e)\big|\bigr)
\lambda(de) (ds) < +\infty \quad\mbox{a.s.}
\]
\item$\cH$ (resp. $\cH(0,T)$) is the set of all predictable
processes on $\R^+$ (resp. on $[0,T]$). $\L^2_{\mathit{loc}}(W)$ is the
subspace of $\cH$ such that, for any $t\geq0$,
\[
\int_0^t |Z_s|^2 ds <
+\infty\quad\mbox{a.s.}
\]
\item\textbf{$\M_{\mathit{loc}}$} is the set of \cadlag local
martingales orthogonal to $W$ and $\widehat{\pi}$. If
$M\in\M_{\mathit{loc}}$, then
%
\begin{equation}
\bigl[M,W^i\bigr]_t=0, \ 1\le i\le d, \quad\mbox{and}\quad \bigl[M,\widehat\pi (A,\cdot)\bigr]_t\xch{=0,}{=0}
\end{equation}
for all $A\in\cU$. In other words, $\E(\Delta M*\pi\mid\cP\otimes
\cU)=0$, where the product $*$ denotes the integral process (see
II.1.5 in \cite{JS03}).
\item\textbf{$\M$} is the subspace of $\M_{\mathit{loc}}$ of martingales.
\item$\cV$ 
is the set of all \cadlag
progressively measurable processes $R$
of finite variation such that $R_0=0$.
\end{itemize}
For a given process $R\in\cV$, we denote by $|R|_t$ the variation of
$R$ on $[0,t]$ and by $dR$ the random measure generated by its trajectories.
By $\cT$ we denote the set of all stopping times with values in
$[0,T]$ and by $\cT_t$ the set of all
stopping times with values in $[t,T]$. We say that a sequence $(\t
_k)_{k\in\N}\subset T$ is stationary if
$\P(\liminf_{k\to+\infty}\{\t_k=T\})=1$. For $X\in\cS$, we set
$X_{t-}=\Lim_{s\nearrow t}X_s$ and $\Delta X_t=X_t-X_{t-}$ with the
convention that $X_{0-}=0$.

Now, since we are dealing with a general filtration, we recall Lemma
III.4.24 in \cite{JS03}, which gives the representation property
of a local martingale in our context.
\begin{lemma} \label{loc-mart-decomposition}
Every local martingale has a decomposition
\[
\int_0^. Z_s dW_s + \int
_0^. \int_\cU V_s(e)
\widehat{\pi}(de,ds) + M,
\]
where $M \in\M_{\mathit{loc}}$, $Z \in\L^2_{\mathit{loc}}(W)$, and $V \in G_{\mathit{loc}}(\pi)$.
\end{lemma}
The Euclidean norm of a vector $y\in\R^k$ will be defined by
$|y|=\sqrt{\sum_{i=1}^k|y_i|^2}$, and for any $k\times d$ matrix $z$,
we define $|z|=\sqrt{\operatorname{Trace}(zz^t)}$, where $z^t$ stands for the
transpose of $z$. The quadratic variation of a martingale $M\in\R^k$
is defined by $[M]_t=\Sum_{i=1}^k[M^i]_t$. By $[M]^c$ we denote the
continuous part of the quadratic variation $[M]$. Let us introduce the
following spaces of processes for any real constant $p\ge2$:
\begin{itemize}
\item\textbf{$\L^p$} is the space of $\mathbb{R}^k$-valued random
variables $\xi$ such that
\[
\|\xi\|_{\L^p}:=E \bigl[|\xi|^{p} \bigr] ^{1/p}<+
\infty.
\]

\item\textbf{$\cS^p$} is the space of $\mathbb{R}^k$-valued, $\cF
_t$-adapted, and \cadlag
processes $ ( Y_{t} ) _{0\leq t\leq T}$ such that
\begin{equation*}
\|Y\|_{\cS^p}:=E \Bigl[ \sup_{0\leq t\leq T}|Y_{t}|^{p}
\Bigr]^{1/p} <+\infty.
\end{equation*}

\item\textbf{$\M^{p}$} is the set of all $\mathbb R^k$-valued
martingales $M\in\M$ such that $\E([M]_T)^\frac{p}{2}<+\infty$.
\item\textbf{$\cM^{p}$} is the set of $\mathbb{R}^{k\times
d}$-valued and
$\cF$-progressively measurable processes\break $(Z_{t})_{0\leq t\leq T}$
such that
\begin{equation*}
\|Z\|_{\mathcal{M}^{p}}:=E \Biggl[ \Biggl(\int_{0}^{T}|Z_{s}|^{2}ds
\Biggr)^\frac{p} {2} \Biggr]^{1/p} <+\infty.
\end{equation*}

\item\textbf{$\cL^p$} is the set of $\cP\otimes\cU$-measurable
mappings $V:\varOmega\times[0,T]\times U\rightarrow\R^k$
such that
\begin{equation*}
\big\|V(e)\big\|_{\mathcal{L}^{p}}:=E \Biggl[ \Biggl(\int_{0}^{T}
\int_U\big|V_{s}(e)\big|^{2}\lambda(de) ds
\Biggr)^\frac{p} {2} \Biggr] ^{1/p}<+\infty.
\end{equation*}

\item\textbf{$\cL_\lambda^p$} is the set of measurable functions
$\phi:U\rightarrow\R^k$ such that
\begin{equation*}
\big\|\phi(e)\big\|_{\mathcal{L}_\lambda^{p}}:= \biggl(\int_U\big|\phi
(e)\big|^{p}\lambda(de) \biggr)^{1/p}<+\infty.
\end{equation*}

\item$\varXi^p$ is the space $\cS^p\times\cM^p\times\cL^p\times\M^p$.

\item$\cV^p$ is the set of all processes $R\in\cV$ such that
$\|R\|_{\cV^{p}}:=\E (|R|^p_T )^{1/p}<\infty$, where
$|R|_T$ denotes the total variation of $R$ on $[0,T]$.
\end{itemize}
In what follows, let $\xi$ be an $\R^k$-valued and $\cF
_T$-measurable random variable, and let $R$ be a process in $\cV$.
Finally, let us consider
a random function $f:[0,T]\times\varOmega\times\R^k\times\R^{k\times
d}\times\cL_\lambda^2\longrightarrow\R^k$
measurable with respect to $\cP\otimes{\cB}
(\R^k)\otimes{\cB}(\R^{k\times d})\otimes\cB(\cL^2_\lambda)$.
In the paper, we consider the following hypotheses:
\begin{enumerate}[$(H1)$]
\item[$(H1)$]$\E [|\xi|^{p}+ (\int_0^T|f(t,0,0,0)|dt )^p+|R|^p_T ]<+\infty$.
\item[$(H2)$] For every $(t,z,v)\in[0,T]\times\R^{k\times d}\times
\cL_\lambda^2$, the mapping $y\in\R^k \to f(t,y,z,v)$ is continuous.
\item[$(H3)$] There exists $\mu\in\R$ such that
\begin{equation}
\nn \bigl(f(t,y,z,v)-f\bigl(t,y',z,v\bigr)\bigr)
\bigl(y-y'\bigr)\le\xch{\mu\bigl(y-y'\bigr)^2,}{\mu\bigl(y-y'\bigr)^2}
\end{equation}
for all
$t\in[0,T], y, y'\in\R^k, z\in\R^{k\times d}, v\in\cL^2_\lambda$.
\item[$(H4)$] For every $r>0$, the mapping $t\in[0,T]\to\Sup_{|y|\le
r}|f(t,y,0,0)-f(t,0,\break 0,0)|$ belongs to $\L^1(\varOmega\times[0,T])$.
\item[$(H5)$] $f$ is Lipschitz continuous w.r.t. $z$, that is, there
exists a constant $L>0$ such that
\begin{equation}
\nn \big|f(t,y,z,v)-f\bigl(t,y,z',v\bigr)\big|\le \xch{L\big|z-z'\big|,}{L|z-z'|}
\end{equation}
for all $t\in[0,T], y\in\R^k, z, z'\in\R^{k\times d}, v\in\cL
^2_\lambda$.
\item[$(H6)$] $f$ is Lipschitz continuous w.r.t. $v$, that is, there
exists a constant $L>0$ such that\vadjust{\eject}
\begin{equation}
\nn\big|f(t,y,z,v)-f\bigl(t,y,z,v'\bigr)\big|\le \xch{L\big\|v-v'\big\|_{\cL_\lambda^2},}{L\|v-v'\|_{\cL_\lambda^2}}
\end{equation}
for all $t\in[0,T], y\in\R^k, z\in\R^{k\times d}, v, v'\in\cL^2_\lambda$.
\end{enumerate}

To begin with, let us make precise the notion of $\L^p$-solutions of
the \hbox{GBSDE}~\eqref{RBSDE-jumps},
which we consider throughout this paper.
\begin{defin}\label{Lp-solution-GBSDE-definition}
We say that $(Y,Z,V,M):=(Y_t,Z_t,V_t,M_t)_{0\le t\le T}$ is an $\L
^p$-solution of the GBSDE \eqref{RBSDE-jumps} if
$(Y,Z,V,M)\in\varXi^p$ \xch{and}{and Eq.}~\eqref{RBSDE-jumps} is satisfied $\P$-a.s.
\end{defin}

\section{Generalized BSDEs with constant terminal time}

\subsection{$\L^2$-solutions}

In this subsection, we study the classical case of $\L^2$-solutions of
GBSDE \eqref{RBSDE-jumps}. The results given here generalize those of
\cite{P97} and \cite{KP14}. Note that the integrability condition
$(H1)_{p=2}$ made on $f(\cdot,0,0,0)$ is weaker than the assumption
$E\int_0^T|f(t,0,\break 0,0)|^2dt<+\infty$, $t\in[0,T]$, made in those
papers, which means that our assumption is weaker than that of \cite
{KP14} even in the case $R\equiv0$.

Let us begin by giving
nonstandard a priori estimates on the solution, which will play a
primordial role in the proof of Theorem~\ref
{existence-uniqueness-in-L2-multi-case}.
Let us first make the following assumption:
\begin{itemize}
\item[$(A)$] There exist $L\ge0$, $\mu\in\R$, and a nonnegative
progressively measurable process $\{f_t\}_{t\in[0,T]}$ satisfying $\E
(\int_0^{T}f_s ds)^2<+\infty$ such that

$\forall(t,y,z,v)\in[0,T]\times\R^k\times\R^{k\times d}\times\cL
^2_\lambda$,
\begin{equation}
\nn \operatorname{\widehat{sgn}}(y)f(t,y,z,v)\le f_t+\mu|y|+L|z|+L\|v
\|_{\cL^2_\lambda},\quad dt\otimes d\P\text{-a.s.}
\end{equation}
\end{itemize}
\begin{remark}
Note that $(A)$ is not a new assumption, but a direct consequence of
assumptions $(H3)$, $(H5)$, and $(H6)$ with $f_t=|f(t,0,0,0)|$.
In fact, three assumptions $(H3)$, $(H5)$, and $(H6)$ are reduced to a
single one \textup{(}assumption~$(A))$ for simplicity.
\end{remark}
\begin{lemma}\label{Lemma9} Let assumption $(A)$ hold, and let
$(Y,Z,V,M)$ be a solution of\break \mbox{GBSDE~\eqref{RBSDE-jumps}}. If $Y\in\cS^2$
and
%
\begin{equation}
\label{f_s-2-integrability} \E|\xi|^2+\E\Biggl(\int_0^{T}f_s
ds\Biggr)^2+\E|R|^2_T<+\infty,
\end{equation}
then, $(Z,V,M)$ belongs to $\cM^2\times\cL^2\times\M^2$, and for
some $a\ge\mu+2L^2$, there is a constant $C>0$
such that, for all $0\le q\le t\le T$,
\begin{align}
\label{exponential-estimate-Z-V-K-L2}
&\E \Biggl[\Sup_{s\in[t,T]}e^{2as}|Y_s|^2+\int_t^{T}e^{2as}|Z_s|^2 ds+\int_{t}^{T}\int_Ue^{2as}\big|V_{s}(e)\big|^2\lambda(de) ds\nn\\
&\qquad +e^{2aT}[ M]_T-e^{2at}[M]_t\Bigm|\cF_q \Biggr]\nn\\
&\quad \le C\E \Biggl[e^{2aT}|\xi|^2+ \Biggl(\int_t^{T}e^{as}f_s ds\Biggr)^2+ \Biggl(\int_t^{T}e^{as}d|R|_s\Biggr)^2\Bigm|\cF_q \Biggr].
\end{align}
\end{lemma}
\begin{proof}
The proof is performed in two steps. For simplicity, we can assume
w.l.o.g.\ that $a=0$. Indeed, let us fix $a\ge\mu+2L^2$ and define
$\widetilde{Y}_t=e^{at}Y_t$, $\widetilde{Z}=e^{at}Z_t$, $\widetilde
{V}_t=e^{at}V_t$, $d\widetilde{M}_t=e^{at}dM_t$. Observe that
$(\widetilde{Y},\widetilde{Z},\widetilde{V},\widetilde{M})$ solves
the following GBSDE:
\begin{align*}
\widetilde{Y}_t&=\widetilde{\xi}+\int^T_t\widetilde {f}(s,\widetilde{Y}_s,\widetilde{Z}_s,\widetilde{V}_s) ds+\int_t^Td\widetilde{R}_s -\int^T_t\widetilde{Z}_s dW_s\\[-1pt]
&\quad -\int_{t}^{T}\int_U\widetilde{V}_{s}(e)\widehat\pi(de,ds)-\int_t^Td\widetilde{M}_s,\quad t\in[0,T],
\end{align*}
where $\widetilde{\xi}=e^{aT}\xi$, $\widetilde
{f}(t,y,z,v)=e^{at}f(t,e^{-at}y,e^{-at}z,e^{-at}v)-ay$, and
$d\widetilde{R}_t=e^{at}dR_t$. Notice that $\widetilde{f}$ satisfies
assumption $(A)$ with $\widetilde{f}_t=e^{at}f_t$, $\widetilde{\mu
}=\mu-a$, $\widetilde{L}=L$. Since we are working on a compact time
interval, the integrability conditions are equivalent with or without
the superscript \ $\widetilde{}$ . Thus, with
this change of variable, we reduce to the case $a=0$ and $\mu+2L^2\le
0$. We omit the
superscript \ $\widetilde{}$ \ for notational convenience.

{\bf Step 1.} First, we show that there exists a constant $C>0$ such
that, for all $0\le q\le t\le T$,
\begin{align}
\label{estimate-YZVM-L2-step1}
& \E \Biggl(\int_t^T|Z_s|^2 ds+\int_t^T\int_U\big|V_{s}(e)\big|^2 \lambda (de) ds +\int_t^Td[M]_s\Bigm| \cF_q \Biggr)\nn           \\[-2pt]
&\quad \le C\E \Biggl(\Sup_{u\in[t,T]}|Y_u|^2+ \Biggl( \int_t^Tf_s ds \Biggr)^2+ \Biggl(\int_t^Td|R|_s \Biggr)^2\Bigm|\cF_q \Biggr).
\end{align}
Since there is a lack of integrability of the processes $(Z,V,M)$, we
are proceeding by localization. For $n\in\N$, we set
\[
\tau_n=\inf\Biggl\{t>0;\int_0^t|Z_s|^2ds+
\int_{0}^{t}\int_U\big|V_{s}(e)\big|^2
\lambda(de) ds+[M]_t>n\Biggr\}\wedge T.
\]
By It\^o's formula (see \cite[Thm.~II.32]{Protter90}),
\begin{align}
\label{rbsde-ito}
&|Y_{t\wedge{\t_n}}|^2+\int_{t\wedge{\t_n}}^{\t_n}|Z_s|^2 ds+\int_{t\wedge{\t_n}}^{\t_n}\int_U\big|V_{s}(e)\big|^2\pi(de,ds)+\int_{t\wedge{\t_n}}^{\t_n}d[M]_s\nn\\[-2pt]
&\quad =|Y_{\t_n}|^2+2\int_{t\wedge{\t_n}}^{\t_n}Y_{s}f(s,Y_s,Z_s,V_s) ds+2\int_{t\wedge{\t_n}}^{\t_n}Y_{s-}dR_s\nn\\[-2pt]
&\qquad -2\int_{t\wedge{\t_n}}^{\t_n} Y_{s}Z_s dW_s-2\int_{t\wedge{\t_n}}^{\t_n}\int_U Y_{s-}V_{s}(e)\widehat{\pi}(de,ds)-2\int_{t\wedge{\t_n}}^{\t_n}Y_{s-}dM_s .
\end{align}
But from $(A)$, the basic inequality $2ab\le2a^2+\frac{b^2}{2}$, and
the fact that $\mu+2L^2\le0$ we have that
\begin{align}
\label{gbsde-ito1-L2}
2Y_{s}f(s,Y_s,Z_s,V_s)&\le 2L|Y_{s}||Z_{s}|+2L|Y_{s}|\big\|V_{s}(e)\big\| _{\cL^2_\lambda}+2\mu|Y_{s}|^2+2|Y_{s}|f_s\nn\\[-2pt]
&\le 2\bigl(\mu+2L^2\bigr)|Y_{s}|^2+2|Y_{s}|f_s+\frac{1}{2}|Z_{s}|^2+\frac{1}{2}\int_U\big|V_{s}(e)\big|^2\lambda(de)\nn\\[-2pt]
&\le 2|Y_{s}|f_s+\frac{1}{2}|Z_{s}|^2+\frac{1}{2}\int_U\big|V_{s}(e)\big|^2\lambda(de).
\end{align}
Then, plugging the last inequality into \eqref{rbsde-ito}, we deduce
\begin{align*}
&\frac{1}2\int_{t\wedge{\t_n}}^{\t_n}|Z_s|^2 ds+\int_{t\wedge{\t_n}}^{\t_n}\int_U\big|V_{s}(e)\big|^2\pi(de,ds)\nn\\[-1pt]
&\qquad -\frac{1}2\int_{t\wedge{\t_n}}^{\t_n}\int_U\big|V_{s}(e)\big|^2 \lambda (de) ds+\int_{t\wedge{\t_n}}^{\t_n}d[M]_s\nn\\[-1pt]
&\quad \le \sup_{s\in[t\wedge\t_n,T]}|Y_{s}|^2+2\sup_{s\in[t\wedge\t_n,T]}|Y_{s}| \Biggl(\int_{t\wedge{\t_n}}^Tf_s ds \Biggr)\\[-1pt]
&\qquad +2\sup_{s\in[t\wedge\t_n,T]}|Y_{s}| \Biggl(\int_{t\wedge{\t_n}}^{T}d|R|_s \Biggr)-2\int_{t\wedge{\t_n}}^{\t_n} Y_{s}Z_s dW_s\\[-1pt]
&\qquad -2\int_{t\wedge{\t_n}}^{\t_n}\int_U Y_{s-}V_{s}(e)\widehat{\pi }(de,ds)-2\int_{t\wedge{\t_n}}^{\t_n}Y_{s-} dM_s.\nn
\end{align*}
Hence, using the inequality $2ab\le a^2+b^2$, we obtain
\begin{align}
\label{gbsde-ito2-L2}
&\!\!\!\frac{1}2\int_{t\wedge{\t_n}}^{\t_n}|Z_s|^2 ds\,{+}\,\int_{t\wedge{\t_n}}^{\t_n}\int_U\big|V_{s}(e)\big|^2\pi(de,ds)\nn\\[-1pt]
&\!\!\!\qquad -\frac{1}2\int_{t\wedge{\t_n}}^{\t_n}\int_U\big|V_{s}(e)\big|^2 \lambda (de) ds+\int_{t\wedge{\t_n}}^{\t_n}d[M]_s\nn\\[-1pt]
&\!\!\!\quad \le 3\sup_{u\in[t\wedge\t_n,T]}|Y_{u}|^2+ \Biggl(\int_{t\wedge{\t_n}}^Tf_s ds\Biggr)^2+ \Biggl(\int_{t\wedge{\t_n}}^{T}d|R|_s\Biggr)^2\nn\\[-1pt]
&\!\!\!\qquad \,{-}\,2\!\int_{t\wedge{\t_n}}^{\t_n} Y_{s}Z_s dW_s\,{-}\,2\!\int_{t\wedge{\t_n}}^{\t_n}\int_U Y_{s-}V_{s}(e)\widehat{\pi}(de,ds)\,{-}\,2\!\int_{t\wedge{\t_n}}^{\t_n}Y_{s-}dM_s.
\end{align}
Note that since $Y\in\cS^2$, by the definition of the stopping time
$\t_n$ it follows by the BDG inequality that
$\int_0^{t\wedge\t_n}Y_{s}Z_s dW_s$,
$\int_0^{t\wedge\t _n}\!\int_U Y_{s-}V_s(e)\widehat{\pi}(de,ds)$
and $\int_0^{t\wedge\t _n}Y_{s-} dM_s$ are uniformly integrable martingales.
Consequently, taking the conditional expectation w.r.t. $\cF_q$, $0\le
q\le t\le T$, in both sides of \eqref{gbsde-ito2-L2} yields
\begin{align}
&\E \Biggl(\frac{1}2\int_{t\wedge{\t_n}}^{\t_n}|Z_s|^2 ds+\frac{1}2\int_{t\wedge{\t_n}}^{\t_n}\int_U\big|V_{s}(e)\big|^2\lambda(de) ds +\int_{t\wedge{\t_n}}^{\t_n}d[M]_s\Bigm|\cF_q\Biggr)\nn\\[-1pt]
&\quad \le \E \Biggl(3\sup_{u\in[t\wedge\t_n,T]}|Y_{u}|^2+\Biggl(\int_{t\wedge{\t_n}}^Tf_s ds\Biggr)^2+ \Biggl(\int_{t\wedge{\t_n}}^{T}d|R|_s\Biggr)^2\Bigm|\cF_q \Biggr).\nn
\end{align}
Therefore, letting $n$ to infinity and using Fatou's lemma, we obtain
\eqref{estimate-YZVM-L2-step1}.

{\bf Step 2.} In this step, we will estimate $\E (\Sup_{u\in
[t,T]}|Y_u|^2|\cF_q )$, $0\le q\le t\le T$. Applying It\^o's
formula to $|Y_{t}|^2$ for each $t\in[0,T]$, we get
\begin{align}
\label{rbsde-ito-L2-Step2}
&|Y_{t}|^2+\int_{t}^{T}|Z_s|^2 ds+\int_{t}^{T}\int_U\big|V_{s}(e)\big|^2\pi(de,ds)+\int_{t}^{T}d[M]_s\nn\\[-1pt]
&\quad =|\xi|^2+2\int_{t}^{T}Y_{s}f(s,Y_s,Z_s,V_s) ds+2\int_{t}^{T}Y_{s-}dR_s\nn\\[-1pt]
&\qquad -2\int_{t}^{T} Y_{s}Z_s dW_s-2\int_{t}^{T}\int_U Y_{s-}V_{s}(e)\widehat{\pi}(de,ds)-2\int_{t}^{T}Y_{s-}dM_s ,
\end{align}
but in view of \eqref{gbsde-ito1-L2}, we deduce
\begin{align}
\label{rbsde-ito-2.73}
&|Y_{t}|^2+\frac{1}2\int_t^T|Z_s|^2 ds+\int_{t}^{T}\int_U\big|V_{s}(e)\big|^2\pi(de,ds)+\int_t^Td[M]_s\nn\\[-2pt]
&\quad \le|\xi|^2+2\int_t^T|Y_{s}|f_s ds+2\int_t^T|Y_{s-}|d|R|_s+\frac{1}2\int_t^T\int_U\big|V_{s}(e)\big|^2 \lambda(de) ds\nn\\[-2pt]
&\qquad -2\int_t^T Y_{s}Z_s dW_s-2\int_{t}^{T}\int_U Y_{s-}V_{s}(e)\widehat{\pi}(de,ds)-2\int_t^TY_{s-}dM_s.
\end{align}
Recalling that $Y\in\cS^2$, thanks to the first step (estimate \eqref
{estimate-YZVM-L2-step1}), it follows that $Z\in\cM^2$, $V\in\cL
^2$, and $M\in\M^2$. Therefore, by the BDG inequality we deduce that
$\int_0^{t}Y_{s}Z_s dW_s$, $\int_{0}^{t}\int_U
Y_{s-}V_{s}(e)\widehat{\pi}(de,ds)$, and $\int_0^{t}Y_{s-} dM_s$
are uniformly integrable martingales.
Therefore, taking the conditional expectation in \eqref
{rbsde-ito-2.73} w.r.t. $\cF_q$, it follows that, for all $0\le q\le
t\le T$,
\begin{align}
\label{ineq2.78}
\E \Biggl(\frac{1}2\int_t^T|Z_s|^2\,\xch{ds}{ds|}+\frac{1}2\int_t^T\int_U\big|V_{s}(e)\big|^2 \lambda(de)ds+\int_t^Td[M]_s\Bigm|\cF_q\Biggr)\le\E (\varsigma\mid \cF_q ),
\end{align}
where $\varsigma=|\xi|^2+2\int_t^T|Y_{s}|f_s ds+2\int_t^T|Y_{s-}|d|R|_s$.

Next, we deduce from \eqref{rbsde-ito-2.73} that
\begin{align}
\Sup_{u\in[t,T]}|Y_{u}|^2&\le\varsigma+
\frac{1}2\int_t^T\int
_U\big|V_{s}(e)\big|^2 \lambda(de) ds +2
\sup_{u\in[t,T]}\Bigg|\int_u^T
Y_{s}Z_s dW_s\Bigg|\nn\\[-1pt]
&\quad +2\sup_{u\in[t,T]}\Bigg|\int_u^T
\int_UY_{s-}V_{s}(e)\widehat{\pi
}(de,ds)\Bigg|+2\sup_{u\in[t,T]}\Bigg|\int_u^TY_{s^-}
dM_s\Bigg|.
\end{align}
Consequently, taking the conditional expectation w.r.t. $\cF_q$, $0\le
q\le t\le T$, we obtain from \eqref{ineq2.78} that
\begin{align}
\label{Y-L2-BDG-cond-exp}
&\E \bigl(\Sup_{u\in[t,T]}|Y_{u}|^2\mid\cF_q \bigr)\nn\\[-2pt]
&\quad \le\E \Biggl(2\varsigma+2\sup_{u\in[t,T]}\bigg|\int_u^T Y_{s}Z_s dW_s\bigg|+2\sup_{u\in[t,T]}\bigg|\int_u^T\int_UY_{s-}V_{s}(e)\widehat{\pi}(de,ds)\bigg|\nn\\[-2pt]
&\qquad +2\sup_{u\in[t,T]}\bigg|\int_u^TY_{s^-}dM_s\bigg|\Bigm|\cF_q \Biggr).
\end{align}
Next, applying the BDG inequality to the martingale terms implies that
there exists a constant $C_1>0$ such that
{\allowdisplaybreaks
\begin{align}
&2\E \Biggl(\sup_{u\in[t,T]}\Biggl\llvert \int_u^TY_{s}Z_s dW_s\Biggr\rrvert \Bigm|\cF_q \Biggr)\nn\\[-2pt]
&\quad \le 2C_1\E \Biggl(\sup_{u\in[t,T]}|Y_u|\Biggl(\int_t^T |Z_s|^2ds\Biggr)^{1/2}\Bigm|\cF_q \Biggr)\nn\\
&\quad \le \frac{1}4\E \bigl(\Sup_{u\in[t,T]}|Y_u|^2\Bigm|\cF_q \bigr)+2C_1^2\E \Biggl(\int_t^T|Z_s|^2ds\Bigm|\cF_q \Biggr),\label{ineq2.75}\\[3pt]
&2\E \Biggl(\sup_{u\in[t,T]}\Biggl\llvert \int_u^TY_{s-}V_s(e)\widehat{\pi}(de,ds)\Biggr\rrvert \Bigm|\cF_q \Biggr)\nn\\
&\quad \le 2C_1\E \Biggl(\sup_{u\in[t,T]}|Y_u|\Biggl(\int_t^T \big|V_s(e)\big|^2\pi(de,ds) \Biggr)^{1/2}\Bigm|\cF_q \Biggr)\nn\\
&\quad \le \frac{1}4\E \Bigl(\Sup_{u\in[t,T]}|Y_u|^2\Bigm|\cF_q \Bigr)+2C_1^2\E \Biggl(\int_t^T\big|V_s(e)\big|^2\lambda(de)ds\Bigm|\cF_q \Biggr),\label{ineq2.76}
\end{align}}%
and
\begin{align}
\label{ineq2.77}
&2\E \Biggl(\sup_{u\in[t,T]}\Biggl\llvert \int_u^TY_{s-}dM_s\Biggr\rrvert \Bigm|\cF_q \Biggr)\nn\\
&\quad \le2C_1\E \Biggl(\sup_{u\in[t,T]}|Y_u|\Biggl(\int_t^T d[M]_s\Biggr)^{1/2}\Bigm|\cF_q \Biggr)\nn\\
&\quad \le \frac{1}4\E \Bigl(\Sup_{u\in[t,T]}|Y_u|^2\Bigm|\cF_q \Bigr)+2C_1^2\E \Biggl(\int_t^Td[M]_s\Bigm|\cF_q\Biggr).
\end{align}
Hence, plugging estimates \eqref{ineq2.75}--\eqref{ineq2.77} into
\eqref{Y-L2-BDG-cond-exp} implies in view of \eqref{ineq2.78} that
there exits a constant $C_2>0$ such that
\begin{equation}
\nn \E \Bigl(\Sup_{u\in[t,T]}|Y_u|^2\Bigm|\cF_q \Bigr)\le C_2\E(\varsigma \mid\cF_q).
\end{equation}
Applying Young's inequality yields
\begin{align}
\label{ineq2.79}
&C_2\E \Biggl(\int_t^T|Y_{s}|f_s ds+\int_t^T|Y_{s-}|d|R|_s\Bigm|\cF _q \Biggr)\nn\\
&\quad \le \frac{1}4\E \Bigl(\Sup_{u\in[t,T]}|Y_u|^2\Bigm|\cF_q \Bigr)+C_3\E \Biggl[ \Biggl(\int_t^Tf_s ds \Biggr)^2 +\Biggl(\int_t^Td|R|_s\Biggr)^2\Bigm|\cF_q \Biggr],
\end{align}
from which we deduce, coming back to the definition of $\varsigma$,
that there exists $C_4>0$ such that
\begin{eqnarray*}
\E \Bigl(\Sup_{u\in[t,T]}|Y_u|^2\Bigm|\cF_q
\Bigr) \le C_4\E \Biggl(|\xi|^2+ \Biggl(\int
_t^Tf_s ds \Biggr)^2+
\Biggl(\int_t^Td|R|_s
\Biggr)^2\Bigm|\cF_q \Biggr).
\end{eqnarray*}
Finally, combining this with \eqref{estimate-YZVM-L2-step1}, the
desired result follows, which ends the proof.
\end{proof}

Now, we give the main result of this subsection.
\begin{thm}{$(L^2$-solutions\textup{)}}\label
{existence-uniqueness-in-L2-multi-case}
Assume that $(H1)_{p=2}$--$(H6)$ are in force. Then, there exists a
unique $\L^2$-solution $(Y,Z,V,M)$
for the GBSDE \eqref{RBSDE-jumps}.\vadjust{\eject}
\end{thm}
\begin{proof}
{\it Uniqueness}.
Let $(Y,Z,V,M)$ and $(Y',Z',V',M')$ denote respectively two $\L
^2$-solutions of GRBSDE \eqref{RBSDE-jumps}.
Define $(\bar{Y},\bar{Z},\bar{V},\bar{M})=(Y-Y',Z-Z',V-V',M-M')$.
Then, $(\bar{Y},\bar{Z},\bar{V},\bar{M})$ solves the following
GBSDE in $\varXi^2$:
\begingroup
\abovedisplayskip=5pt
\belowdisplayskip=5pt
\begin{align}
\label{uniqueness-p=2}
\bar{Y}_t&=\int^T_t\bigl(f(s,Y_s,Z_s,V_s)-f\bigl(s,Y'_s,Z'_s,V'_s\bigr)\bigr) ds -\int^T_t\bar Z_s dW_s\\
&\quad -\int_{t}^{T}\int_U\bar V_{s}(e)\widehat{\pi}(de,ds)-\int_t^Td\bar M_s,\quad t\in[0,T].\nn
\end{align}
It follows from $(H3)$, $(H5)$, and $(H6)$ that
\begin{align*}
&\operatorname{\widehat{sgn}}(\bar{y}) \bigl(f(t,y,z,v)-f\bigl(t,y',z',v'\bigr)\bigr)\\
&\quad =\operatorname{\widehat{sgn}}(\bar{y}) \bigl[f(t,y,z,v)-f\bigl(t,y',z,v\bigr)+f\bigl(t,y',z,v\bigr)-f\bigl(t,y',z',v \bigr)\nn\\
&\qquad +f\bigl(t,y',z',v\bigr)-f\bigl(t,y',z',v'\bigr) \bigr]\nn\\
&\quad \le \mu|y|+L|z|+L\|v\|_{\cL^2_\lambda},
\end{align*}
which means that assumption $(A)$ is satisfied for the generator of
GBSDE \eqref{uniqueness-p=2} with $f_t\equiv0$. By Lemma~\ref{Lemma9}
with $q=t=0$ we obtain immediately that $(\bar{Y},\bar{Z},\bar
{V},\bar{M})=(0,0,0,0)$. The proof of the uniqueness is then
complete.

{\it Existence}.
Before giving the proof of the existence part, we will talk a little
bit about it, but let us first give the following assumption, the
so-called general growth condition, which will be needed later:
%
\begin{equation}
(H_{gg})\quad\text{For\ every}\ (t,y)\in[0,T]\times\R,\quad  \big|f(t,y,0,0)\big|
\le\big|f(t,0,0,0)\big|+\gamma\big(|y|\big),
\end{equation}
\endgroup
where $\gamma:\R^+\to\R^+$ is a deterministic continuous increasing
function.

The proof method of Theorem~\ref
{existence-uniqueness-in-L2-multi-case} is enlightened by \cite{KR13,P97,P99,briandetal03,KP14}, but
of course with some obvious changes. More precisely, the first step
uses arguments given in \cite{KR13,P97,KP14}, whereas
the techniques used in the second step, the convolution and weak
convergence, are borrowed from \cite{P97}. The truncation techniques
applied in the third and fourth steps are taken partly from \cite
{briandetal03,KP14,KR13}. However, it should be
mentioned that since we have changed, compared to \cite{KP14}, the $\L
^2$-integrability condition of $f(t,0,0,0)$ from $\E\int_0^T|f(t,0,0,0)|^2 dt<+\infty$ to the one given in $(H1)_{p=2}$ and
due also to the finite-variation part $dR$, some new troubles come up,
especially, when we want to prove an analogous result of \cite[Lemma
4]{KP14}, which says that whenever the data is bounded, so is the
solution of the GBSDE. Their approach fails in our context. This is the
reason why we give nonstandard estimates in Lemma~\ref{Lemma9},
which allows us to overcome this problem (see, e.g., estimates \eqref
{r-estimate} and \eqref{r-estimate-2}). Additionally, in order to
prove the existence part of Theorem~\ref
{existence-uniqueness-in-L2-multi-case}, we first need an existence
result under the assumptions of this theorem but with $(H4)$ replaced
with $(H_{gg})$, which extends results given in \cite{P97} and \cite
{KP14}.

The proof is divided into five steps as follows. Note that we will
frequently apply Lemma~\ref{Lemma9}. For simplicity, we will assume
w.l.o.g. that $a=0$, which means that, in this case, $\mu<0$ (since
$\mu+2L^2\le a=0$, i.e., $\mu\le-2L^2<0$). In the rest of the proof,
we will assume that $\mu<0$.

{\bf Step \xch{1.}{1:}} We first assume additionally that there exists a constant
$l>0$ such that
%
\begin{equation}
\label{lipschitz-assumption} \big|f(t,y,z,v)-f\bigl(t,y',z,v\bigr)\big|\le \xch{l\big|y-y'\big|,}{l|y-y'|}
\end{equation}
for $t\in[0,T], y,y'\in\R^k, z\in\R^{k\times d},v\in\cL
^2_\lambda$. Moreover, we assume also that there exists a constant
$\epsilon>0$ such that
%
\begin{equation}
\label{boundedness-data} |\xi|+\Sup_{t\in[0,T]}\big|f(t,0,0,0)\big|+|R|_T\le
\epsilon.
\end{equation}
For $(\varGamma,\varUpsilon,\varPsi,N)\in\varXi^2$, in view of the assumptions
made on
$\xi$, $f$, and $R$,
define the processes $(Y,Z,V,M)$ as follows:
\begin{align}
Y_t&=\E \Biggl[\xi+\int_0^Tf(s,
\varGamma_s,\varUpsilon_s,\varPsi_s)  ds+
\int_0^TdR_s\biggm|\cF_t
\Biggr]\nn
\\
&\quad -\int_0^tf(s,
\varGamma_s,\varUpsilon_s,\varPsi_s) ds-
\int_0^tdR_s,\nn
\end{align}
and the local martingale
\begin{equation}
\nn \E \Biggl[\xi+\int_0^Tf(s,
\varGamma_s,\varUpsilon_s,\varPsi_s) ds+
\int_0^TdR_s\biggm|\cF_t
\Biggr]-Y_0,
\end{equation}
which thanks to the martingale representation theorem (see Lemma~\ref
{loc-mart-decomposition}), can be decomposed as follows:
\[
\int_0^tZ_s dW_s+
\int_{0}^{t}\int_UV_{s}(e)
\widehat{\pi}(de,ds)+M_t,
\]
where $Z$, $V$, and $M$ belong respectively to $\L^2_{\mathit{loc}}(W)$,
$G_{\mathit{loc}}(\pi)$, and $\M_{\mathit{loc}}$.
Therefore, $(Y,Z,V,M)$ is the unique solution of the GBSDE
\begin{align}
\nn Y_t&=\xi+\int^T_tf(s,
\varGamma_s,\varUpsilon_s,\varPsi_s) ds+
\int_t^T  dR_s-\int
^T_tZ_s dW_s
\\
&\quad -\int_{t}^{T}\int
_UV_{s}(e)\widehat{\pi}(de,ds)-\int
_t^T dM_s,\quad t\in[0,T].
\end{align}
Moreover, from the conditions on $\xi$, $f$ and $R$ it is easy to
prove that $(Y,Z,V,\break M)\in\varXi^2$.

As a by-product, we may define the mapping $\varPhi:\varXi^2\to\varXi^2$
that associates $(\varGamma,\varUpsilon,\varPsi,N)$ with
$\varPhi((\varGamma,\varUpsilon,\varPsi,N))=(Y,Z,V,M)$. By standard arguments
(see, e.g., the proof of \cite[Thm.~55.1]{P97}) it can be shown that
$\varPhi$ is contractive on the Banach space $\varXi^2$ endowed with the norm
\begin{align*}
\big\|(Y,Z,V,M)\big\|_\beta&=\E \Biggl\{\Sup_{t\in[0,T]}e^{\beta t}|Y_t|^2+\int_0^Te^{\beta t}|Z_t|^2 dt\\
&\quad +\int_0^T\int_Ue^{\beta t}\big|V_t(e)\big|^2\lambda(de) dt\xch{+\Biggl[\int_0^. e^{\beta t} dM_t\Biggr]_T \Biggr\}^\frac{1}{2},}{+\Biggl[\int_0^. e^{\beta t} dM_t\Biggr]_T \Biggr\}^\frac{1}{2}}
\end{align*}
for a suitably chosen constant $\beta>0$. Consequently, $\varPhi$ has a
fixed point $(Y,Z,V,\break M)\in\varXi^2$. Therefore, clearly, $(Y,Z,V,M)$ is the
unique solution of GBSDE \eqref{RBSDE-jumps} under the assumptions
made so far.

{\bf Step \xch{2.}{2:}} In this step, we will show how to dispense with
assumption \eqref{lipschitz-assumption}. We state and prove
the following lemma.\vadjust{\eject}
\begin{lemma}\label{prop-GBSDE-contraction}
Assume that $(H1)_{p=2}, (H2),(H3),(H_{gg}), (H5)$, $(H6)$, and \eqref
{boundedness-data} hold. For given $(\varUpsilon,\varPsi)\in\cM^2\times
\cL^2$,
there exists a unique quadruple of processes $(Y,Z,V,\break M)\in\varXi^2$ such that
\begin{align}
\label{GBSDE-contraction} Y_t&=\xi+\int^T_tf(s,Y_s,
\varUpsilon_s,\varPsi_s) ds+\int_t^T
 dR_s-\int^T_tZ_s
 dW_s\nn
\\
&\quad -\int_{t}^{T}\int
_UV_{s}(e)\widehat{\pi}(de,ds)-\int
_t^T dM_s.
\end{align}
\end{lemma}
For notational convenience, we set $f(t,y)=f(t,y,\varUpsilon_t,\varPsi_t)$
for each $y\in\R^k$.

\smallskip

\noindent{\bf Proof.} Uniqueness is proved by arguing as for uniqueness in
Theorem~\ref{existence-uniqueness-in-L2-multi-case}. The result
follows immediately. For the existence part,
we follow the line of the proof of \cite[Prop.~2.4]{P99}.
Now, let us assume that \eqref{boundedness-data}
holds and define $f_n(t,y)=(\rho_n*f(t,\cdot))(y)$, where $\rho_n:\R
\to\R^+$ is a sequence of smooth functions with compact support
that approximate the Dirac measure at $0$ and satisfy $\int\rho
_n(z)dz=1$. Moreover, they are defined such that
$\varpi$ satisfying $\varpi(r)=\Sup_n\Sup_{|y|\le r}\int_{\R
}\gamma(|y|)\rho_n(y-z) dz$ is finite for all $r\in\R^+$.
Note that $f$ satisfies the following assumptions:
\begin{itemize}
\item[$(\mathit{i})$] $|f(t,y)|\le|f(t,0,0,0|+L(|\varUpsilon_s|+\|\varPsi_s\|_{\cL
_\lambda^2})+\gamma(|y|)$,
\item[$(\mathit{ii})$] $\E\int_0^T|f(t,0)|^2 dt<+\infty$,
\item[$(\mathit{iii})$] $(y-y')(f(t,y)-f(t,y'))\le0$,
\item[$(\mathit{iv})$] $y\rightarrow f(t,y)$ is continuous for all $t$ a.s.
\end{itemize}
Thus, it is elementary to check that $f_n$ satisfies (i)--(iv) with the
same constant~$L$ and $\varpi$
instead of $\gamma$. However, we cannot apply
step 1 of the proof since 
$f_n$ is not necessarily globally Lipschitz continuous in $y$ but only
locally Lipschitz. Hence, to overcome this
problem, we add a truncation function $T_p$ in $f_n$. Indeed, define,
for each $p\in\N$,
\begin{equation}
\nn f_{n,p}(t,y)=f_n\bigl(t,T_p(y)\bigr) \quad
\text{such\ that}\ T_p(y)=\frac
{py}{|y|\vee p}.
\end{equation}
Notice that, for all $n,p\in\N$, $y\to f_{n,p}(t,y)$ is globally
Lipschitz and satisfies the conditions of Step $1$. Therefore, for all
$n,p\in\N$, according to what has already been proved in Step 1,
there exists a unique solution $(Y^{n,p},Z^{n,p},V^{n,p},M^{n,p})$ to GBSDE \eqref{GBSDE-contraction}
associated with $(\xi,f_{n,p}+dR)$. Furthermore, it follows from
\eqref{boundedness-data} and Lemma~\ref{Lemma9} with $a=0$ and $q=t$
that there exists a universal constant $C_1>0$ such that, for all $n,p
\in\N$ and $t\in[0,T]$,
\begin{align}
\label{r-estimate}
&\big|Y^{n,p}_t\big|^2+\E \Biggl(\int_t^T\big|Z^{n,p}_s\big|^2 ds+\int_t^T\int_U\big|V^{n,p}_{s}(e)\big|^2\lambda(de) ds +\int_t^Td\bigl[M^{n,p}\bigr]_s\biggm|\cF_t \Biggr)\nn\\
&\quad \le C_1\E \Biggl[|\xi|^2+ \Biggl(\int_t^T\big|f(s,0,0,0)\big| ds \Biggr)^2+\Biggl(\int_t^Td|R|_s\Biggr)^2\biggm|\cF_t \Biggr]\nn\\
&\quad \le C_1\epsilon^2\bigl(2+T^2\bigr):=r^2.
\end{align}
Hence, for any $p> r$, the sequence $(Y^{n,p},Z^{n,p},V^{n,p},M^{n,p})$
does not depend on $p$. Then we denote it by $(Y^n,Z^n,V^n,M^n)$, and
it is a solution to GBSDE \eqref{GBSDE-contraction} associated with
$(\xi,f_{n}+dR)$. Moreover, now $f_n$
satisfies the conditions of Lemma~\ref{Lemma9} with a constant
independent of $n$, and thus the sequence $(Y^n,U^n,Z^n,V^n,M^n)$ is
uniformly bounded, that is,
\begin{align}
\label{uniform-boundedness-Y-Z-V-M-f-Hilbert}
&\Sup_{n\in\N}\E \Biggl[\int_0^T\big|Y^n_t\big|^2  dt+ \Biggl(\int_0^Tf_n\bigl(t,Y^n_t\bigr) dt \Biggr)^2+\int_0^T\big|Z^n_t\big|^2 dt\nn\\
&\quad +\int_0^T\int_U\big|V_t^n(e)\big|^2\lambda(de) dt+\bigl[M^n\bigr]_T \Biggr]\le C.
\end{align}
Let us set $U^n_t=f_n(t,Y^n_t)$ for $t\in[0,T]$.
Using the previous uniform estimate of the sequence $\{
(Y^n,Z^n,V^n,U^n,)\}_n$ and the Hilbert structure of $\L^2(\varOmega
\times[0,T])\times\cM^2\times\cL^2\times\mathfrak{H}^2$, we
deduce that we can extract subsequences, still denoted $\{n\}$, that
weakly converge to some process $(Y,Z,V,U)$ in $\L^2(\varOmega\times
[0,T])\times\cM^2\times\cL^2\times\mathfrak{H}^2$, where
$\mathfrak{H}^2$ denotes the set of $\cF_t$-progressively measurable
$\R^k$-valued processes $(U_t)_{t\in[0,T]}$ such that
\[
\|U\|_{\mathfrak{H}^2}:= \Biggl\{\E \Biggl[ \Biggl(\int_0^T|U_t|
 dt \Biggr)^2 \Biggr] \Biggr\}^{\frac{1}2}<+\infty.
\]
Now we deal with the convergence of the martingale $M^n$. By estimate
\eqref{uniform-boundedness-Y-Z-V-M-f-Hilbert} it follows that $\Sup
_{n\ge1}\E|M^n_T|^2<\infty$. Thus, there exists a subsequence, still
denoted $\{n\}$, such that $M_T^n$ converges weakly to some random variable
$M_T$ in $\L^2(\varOmega)$. Let $M_t$ denote the martingale with
terminal value $M_T$.

Next, following \cite[Prop.~2.4]{P99}, we deduce, by using the
martingale representation theorem (see Lemma~\ref
{loc-mart-decomposition}) and orthogonality that the following weak
convergence hold for the martingales in $\L^2(\varOmega)$: for each
$t\in[0,T]$,
\begin{align}
&\int_t^T Z^{n}_s
 dW_s\to\int_t^T Z_s
 dW_s, \ \int_t^T\int
_UV^{n}_s(e)\hat\pi(de,ds)\to\int
_t^T\int_UV_s(e)
\hat\pi (de,ds),\nn
\\
& \quad \text{and} \quad M^n_t\to M_t. \nn
\end{align}
Therefore, taking weak limits in the approximating equation,
we get that $(Y,Z,V,M)$ satisfies the GBSDE 
%
\begin{equation}
\nn Y_t=\xi+\int^T_tU_s
 ds+\int_t^T dR_s-\int
^T_tZ_s dW_s-\int
_{t}^{T}\int_UV_{s}(e)
\widehat{\pi}(de,ds)-\int_t^T dM_s.
\end{equation}
Finally, as in \cite[Prop.~2.4]{P99}, we can show that $U_t=f(t,Y_t)$.
This implies that $(Y,Z,V,M)$ solves GBSDE \eqref{GBSDE-contraction}
under the assumptions made so far in this step,
which completes the proof of Lemma~\ref{prop-GBSDE-contraction}.

{\bf Step 3.} In this step, we will show that assumption $(H_{gg})$
assumed so far can be weakened to $(H4)$. In fact, by the mean
of truncation technique we will show that, for given $(\varUpsilon,\varPsi
)\in\cM^2\times\cL^2$ and provided that $(H1)_{p=2}$--$(H6)$ and
\eqref{boundedness-data} hold, GBSDE~\eqref{GBSDE-contraction} has a
solution in $\varXi^2$. The idea behind the proof is to approximate $f$
by a sequence of functions $f_n$ satisfying assumption $(H_{gg})$.
Indeed, let $\theta_r$ be a smooth function such that $0\le\theta
_r\le1$ and satisfies for $r$ large enough:
\begin{equation}
\nn \theta_r(y)= \bigg\{ %
\begin{array}{@{}ll}
1\quad\text{for} \quad|y|\le r,\\
0 \quad\text{for} \quad|y|\ge r+1.
\end{array} %
\end{equation}
Let $\varphi_{r}(t)=\Sup_{|y|\le r}|f(t,y,0,0)-f(t,0,0,0)|\in\L
^1([0,T])$ and, for $n\in\N^*$,
denote $T_n(x)=\frac{xn}{|x|\vee n}$.
The approximation sequence $f_n$ is defined by
\begin{align}
\nn f_n(t,y,\varUpsilon,\varPsi)&=\bigl(f\bigl(t,y,T_n(
\varUpsilon_t),T_n(\varPsi _t)\bigr)-f(t,0,
\varUpsilon_t,\varPsi_t)\bigr)\frac{n}{\varphi_{r+1}(t)\vee n}
\\
&\quad +f(t,0,\varUpsilon_t,\varPsi_t).\nn
\end{align}
We also define a sequence $h_n$ that truncates $f_n$ for $|y|\ge r+1$:
\begin{align}
\nn h_n(t,y,\varUpsilon_t,\varPsi_t)&=
\theta_r(y) \bigl(f\bigl(t,y,T_n(\varUpsilon
_t),T_n(\varPsi_t)\bigr)-f(t,0,
\varUpsilon_t,\varPsi_t)\bigr)\frac{n}{\varphi
_{r+1}(t)\vee n}
\\
&\quad +f(t,0,\varUpsilon_t,\varPsi_t).\nn
\end{align}
Following the same reasoning as in the proof of \cite
[Thm.~4.2]{briandetal03}, it can be shown that $h_n$ still
satisfies the monotonicity condition $(H3)$ but with a positive
constant $C(r,k,n)$ depending on $r,k$, and $n$.
Then, the conditions of Lemma~\ref{prop-GBSDE-contraction} of the
previous step are fulfilled by the data $(\xi,h_n+dR)$. Consequently,
for each $n\in\N^*$,
the GBSDE \eqref{GBSDE-contraction} associated with $(\xi,h_n+dR)$,
admits a unique solution $(Y^n,Z^n,V^n,M^n)\in\varXi^2$. Moreover, since
$yh_n(t,y,\varUpsilon,\varPsi)\le|y|\|f(t,0,0,0)\|_{\infty
}+k|y|(|\varUpsilon|+\|\varPsi\|_{\cL_\lambda^2})$, $h_n$ satisfies the
condition of Lemma~\ref{Lemma9}. Consequently, applying Lemma~\ref
{Lemma9} with $a=0$ and
$q=t=0$, in view of the boundedness assumption \eqref
{boundedness-data}, we get similarly as in \eqref{r-estimate} that,
for each $n\in\N$,
the following estimates hold $d\P\times dt$-a.e.:
%
\begin{equation}
\label{r-estimate-2} \big|Y^{n}_t\big|\le r \quad\text{and} \quad\E
\Biggl(\int_0^T\big|Z^{n}_s\big|^2
 ds+\int_0^T\int_U\big|V^{n}_{s}(e)\big|^2
\lambda(de) ds +\bigl[M^{n}\bigr]_T \Biggr)\le
r^2.
\end{equation}
As a by-product, $(Y^n,Z^n,V^n,M^n)$ is a solution to the GBSDE \eqref
{GBSDE-contraction} associated with the data $(\xi,f_n+dR)$.
Next, we show as in \cite[Thm.~1]{KP14} (see also \cite
[Thm.~4.2]{briandetal03}) by using similar arguments that
$(Y^n,Z^n,V^n,M^n)$ is a Cauchy sequence
in $\varXi^2$, and its limit is $(Y,Z,V,M)\in\varXi^2$.

{\bf Step \xch{4.}{4:}} We now treat the general case. We want to get rid of the
boundedness condition \eqref{boundedness-data} used so far. To this
end, a truncation procedure. Indeed, under assumptions
$(H1)_{p=2}$--$(H6)$, we first set, for each $n\in\N^*$,
\begin{align*}
\xi_n&=T_n(\xi),\qquad f_n(t,y)=f_n(t,y,\varUpsilon_t,\varPsi_t)=f(t,y)-f(t,0)+T_n\bigl(f(t,0)\bigr),\\
R^n_t&=\int_0^t \1_{\{|R|_s\le n\}} dR_s.
\end{align*}
Then, according to the previous step, for each $n\in\N^*$, GBSDE
\eqref{GBSDE-contraction} associated with $(\xi_n,f_n+dR^n)$
has a unique solution $(Y^n,Z^n,V^n,M^n)\in\varXi^2$. Our goal now is to
show that $(Y^n,Z^n,V^n,M^n)$ is a Cauchy sequence in $\varXi^2$. Set
$(\bar Y,\bar Z,\bar V, \bar M)=(Y^m-Y^n,Z^m-Z^n,V^m-V^n,M^m-M^n)$.
Then $(\bar Y,\bar Z,\bar V, \bar M)$ is solution to the GBSDE
\begin{align}
\label{uniqueness-cauchy-p=2} \bar{Y}_t&=\xi_m-\xi_n+\int
_t^Td\bigl(R^m_s-R^n_s
\bigr)+\int^T_t\bigl(f_m
\bigl(s,Y^m_s\bigr)-f_n\bigl(s,Y^n_s
\bigr)\bigr) ds\nn
\\
&\quad -\int^T_t\bar Z_s
 dW_s-\int_{t}^{T}\int
_U\bar V_{s}(e)\widehat{\pi}(de,ds)-\int
_t^Td\bar M_s,\quad t\in[0,T].
\end{align}
Thanks to $(H3)$, $(H5)$, and $(H6)$, the generator of GBSDE \eqref
{uniqueness-cauchy-p=2} satisfies assumption $(A)$ with $f_t\equiv
T_m(f(t,0))-T_n(f(t,0))$
and $L=0$. Therefore, applying Lemma~\ref{Lemma9} with $a=0$ and
$q=t=0$ yields,
for all $n,m\in\N$,
\begin{align}
\label{Truncation-Cauchy-sequence}
&\E \Biggl[\Sup_{t\in[0,T]}|\bar Y_t|^2+\int_0^T|\bar Z_t|^2 dt+\int_{0}^{T}\int_U\big|\bar V_t(e)\big|^2\lambda(de)dt+[\bar M]_T \Biggr]\nn\\
&\quad \le C\E \Biggl[|\xi_m-\xi_n|^2+ \Biggl(\int_0^T\big|T_m\bigl(f(t,0)\bigr)-T_n\bigl(f(t,0)\bigr)\big| dt \Biggr)^2\nn\\
&\qquad + \Biggl(\int_0^Td\big|R^m-R^n\big|_s\Biggr)^2 \Biggr].
\end{align}
Obviously, the right-hand side of \eqref{Truncation-Cauchy-sequence}
tends to $0$ as $n,m\to\infty$. Therefore, $(Y^n,Z^n,\break V^n,M^n)$ is a
Cauchy sequence in $\varXi^2$, and its limit $(Y,Z,V,M)\in\varXi^2$ is an
$L^2$-solution of GBSDE \eqref{GBSDE-contraction}.

{\bf Step \xch{5.}{5:}} In this step, we will finally complete the proof of the
existence part of Theorem~\ref{existence-uniqueness-in-L2-multi-case}.
To this end, we consider a Picard's iteration
procedure. Set $(Y^0,Z^0,V^0,\break M^0)\,{=}\,(0,0,0,0)$ and define $\{
(Y_t^n,Z_t^n,V_t^n,M_t^n)_{t\in[0,T]}\}_{n\ge1}$ recursively in view
of Step $4$, for all $n\ge0$ and $t\in[0,T]$,
\begin{align}
\label{GBSDE-picard} Y^{n+1}_t&=\xi+\int^T_tf
\bigl(s,Y^{n+1}_s,Z^n_s,V^n_s
\bigr) ds+\int_t^TdR_s -\int
^T_tZ^{n+1}_s
 dW_s\nn
\\
&\quad -\int_{t}^{T}\int
_UV^{n+1}_{s}(e)\widehat{\pi }(de,ds)-
\int_t^TdM^{n+1}_s.
\end{align}
From the previous step it follows that, under assumptions
$(H1)_{p=2}$--$(H6)$, for each $n\ge0$, there exists a solution of
GBSDE \eqref{GBSDE-picard}. Let us set
$\delta Y^n:=Y^{n+1}-Y^n$, $\delta Z^n:=Z^{n+1}-Z^n$, $\delta
V^n:=V^{n+1}-V^n$, and $\delta M^n:=M^{n+1}-M^n$. 
$(\delta Y^n,\delta Z^n,\delta V^n,\delta M^n)$ solves the GBSDE
\begin{align*}
\label{GBSDE-picard-diff}
\delta Y^{n}_t&=\int^T_t
\bigl[f\bigl(s,Y^{n+1}_s,Z^n_s,V^n_s
\bigr)-f\bigl(s,Y^{n}_s,Z^{n-1}_s,V^{n-1}_s
\bigr) \bigr] ds-\int^T_t\delta
Z^{n}_s dW_s
\\
&\quad -\int_{t}^{T}\int_U
\delta V^{n}_{s}(e)\widehat\pi(de,ds) -\int
_t^Td\delta M^{n}_s,
\quad t\in[0,T].
\end{align*}
By assumptions $(H3)$, $(H5)$, and $(H6)$ we have that
{\allowdisplaybreaks
\begin{align*}
&\operatorname{\widehat{sgn}}\bigl(\delta y^n\bigr) \bigl(f\bigl(t,y^{n+1},z^n,v^n\bigr)-f\bigl(t,y^{n},z^{n-1},v^{n-1}\bigr)\bigr)\\
&\quad =\operatorname{\widehat{sgn}}\bigl(\delta y^n\bigr) \bigl[\bigl(f\bigl(t,y^{n+1},z^n,v^n\bigr)-f\bigl(t,y^{n},z^{n},v^{n}\bigr)\bigr)\\
&\qquad +\bigl(f\bigl(t,y^{n},z^n,v^n\bigr)-f\bigl(t,y^{n},z^{n-1},v^{n-1}\bigr)\bigr)\bigr]\\
&\quad \le\mu\big|\delta y^n\big|+L\big|\delta z^{n-1}\big|+L\big\|\delta v^{n-1}\big\|_{\cL^2_\lambda},
\end{align*}}%
which is assumption $(A)$ with $f_t= L|\delta z^{n-1}|+L\|\delta
v^{n-1}\|_{\cL^2_\lambda}$ and $L\equiv0$.
Thus, it follows from Lemma~\ref{Lemma9} with $a=q=t=0$ and H\"
{o}lder's inequality that there exists a constant $C>0$ such that
\begin{align*}
&\E \Biggl[ \Sup_{t\in[0,T]} \big|\delta Y^n_t\big|^2+ \int_0^T \big|\delta Z^n_s\big|^2ds +\int_0^T \int_U \big|\delta V^n_s(e)\big|^2 \lambda(de) ds + \bigl[\delta M^n\bigr]_T \Biggr] \nn\\
&\quad \le C \E \Biggl[ \int_0^T L \bigl(\big|\delta Z_t^{n-1}\big|+\big\|\delta V^{n-1}_t(e)\big\|_{\cL^2_\lambda} \bigr) dt \Biggr]^2\\
&\quad \le CL^2T\E \Biggl[\int_0^T\bigl(\big|\delta Z_t^{n-1}\big|^2+\big\|\delta V^{n-1}_t(e)\big\|^2_{\cL^2_\lambda} \bigr) dt\Biggr].
\end{align*}
Consequently, by induction we deduce that, for $n\ge2$,
\begin{align*}
&\E \Biggl[ \Sup_{t\in[0,T]} \big|\delta Y^n_t\big|^2+ \int_0^T \big|\delta Z^n_t\big|^2 dt +\int_0^T \int_U \big|\delta V^n_t(e)\big|^2 \lambda(de) dt + \bigl[\delta M^n\bigr]_T \Biggr] \nn\\
&\quad \le c^{n-1}\E \Biggl[ \int_0^T\big|\delta Z^1_t\big|^2 dt +\int_0^T \int_U \big|\delta V^1_t(e)\big|^2 \lambda(de) dt \Biggr],
\end{align*}
where $c=CL^2T$. Let us first assume, for a sufficiently small $T$,
that $c<1$. Then, since the remaining term of the right-hand side of
the last inequality is finite, we deduce that
$(Y^n,Z^n,V^n,M^n)$ is a Cauchy sequence in $\varXi^2$, and the limit
process $(Y,Z,V,M)$ is a solution to GBSDE \eqref{RBSDE-jumps} in $\varXi^2$.

For the general case, it suffices to subdivide the interval time
$[0,T]$ into a finite number of small intervals, and using the standard
arguments, we can prove the existence of a solution $(Y,Z,V,M)$ of
GBSDE \eqref{RBSDE-jumps} in $\varXi^2$ on the whole interval $[0,T]$.
This completes the proof of this step and thus of the whole proof of
Theorem~\ref{existence-uniqueness-in-L2-multi-case}.
\end{proof}
%
\subsection{Case $p\ge2$}
In this subsection, we study the issue of existence and uniqueness of
$\L^p$-solutions of GBSDE \eqref{RBSDE-jumps} in the case $p\ge2$.
Let us first give a priori estimates for the solution and their
variations induced by a variation of the data.
\begin{lemma} \label{Lp-estimates}
Let assumption $(A)$ hold, and let $(Y,Z,V,M)$ be a solution of GBSDE
\eqref{RBSDE-jumps}. Let us assume moreover that
%
\begin{equation}
\E \Biggl[|\xi|^p+ \Biggl(\int_0^Tf_t
 dt \Biggr)^p+|R|^p_T \Biggr]<+\infty.
\end{equation}
If $Y\in S^p$, then there exists a constant $C_p>0$, depending only on
$p$ and $T$, such that, for every $a\ge\mu+2L^2$,\vadjust{\eject}
\begin{align}
\label{ineq2.62}
&\E \Biggl[ \Sup_{t\in[0,T]} e^{apt}|Y_t|^p+ \Biggl( \int_0^T e^{2at}|Z_t|^2 dt \Biggr)^{p/2}\nn\\[-2pt]
& \qquad + \Biggl( \int_0^T \int_U e^{2at}\big|V_t(e)\big|^2\lambda(de) dt \Biggr)^{p/2} + e^{apT}[M]_T^{p/2}\Biggr] \nn\\[-2pt]
&\quad \le C \E \Biggl[e^{apT}|\xi|^{p} + \Biggl(\int_0^T e^{at}f_t dt\Biggr)^p+ \Biggl(\int_0^Te^{at}d|R|_t\Biggr)^p \Biggr].
\end{align}
\end{lemma}
\begin{proof}
The proof is divided into two steps. By an already used argument (see
Lem\-ma~\eqref{Lemma9}) we can assume w.l.o.g. that $\mu+2L^2\le0$ and
take $a=0$.

{\bf Step 1.} First, we show that
\begin{align}
\label{gbsde-ito2-Lp-pgreaterthan2-T}
&\E \Biggl[ \Biggl(\int_0^{T}|Z_s|^2 ds \Biggr)^\frac{p} {2}+ \Biggl(\int_0^{T}\int_U\big|V_{s}(e)\big|^2 \lambda(de)ds\Biggr)^\frac{p} {2} + \Biggl(\int_0^{T}d[M]_s\Biggr)^\frac{p} {2} \Biggr]\nn\\[-2pt]
&\quad \le C(p,T) \Biggl[\E\Sup_{t\in[0,T]}|Y_{t}|^p+\Biggl(\int_{0}^Tf_s ds\Biggr)^p + \Biggl(\int_0^{T}d|R|_s\Biggr)^p \Biggr].
\end{align}
Indeed, define the sequence of stopping times $\t_n$ for $n\in\N$:
\[
\tau_n=\inf\Biggl\{t>0;\int_0^t|Z_s|^2ds+
\int_{0}^{t}\int_U\big|V_{s}(e)\big|^2
\lambda de(ds)+[M]_t>n\Biggr\}\wedge T.
\]
By It\^o's formula on $|Y_t|^2$,
\begin{align*}
&|Y_0|^2+\int_{0}^{\t_n}|Z_s|^2 ds+\int_{0}^{\t_n}\int_U\big|V_{s}(e)\big|^2\pi(de,ds)+\int_{0}^{\t_n}d[M]_s\\[-2pt]
&\quad =|Y_{\t_n}|^2+2\int_{0}^{\t_n}Y_{s}f(s,Y_s,Z_s,V_s) ds+2\int_{0}^{\t_n}Y_{s-}dR_s-2\int_{0}^{\t_n} Y_{s}Z_s dW_s\\[-2pt]
&\qquad -2\int_{0}^{\t_n}\int_UY_{s-}V_{s}(e)\widehat{\pi}(de,ds)-2\int_{0}^{\t_n}Y_{s-}dM_s.
\end{align*}
But from assumption $(A)$, combined with the inequality $2ab\le\frac
{1}{\varepsilon}a^2+\varepsilon b^2$ for $\varepsilon>0$, since $\mu
+2L^2<0$, we have
\begin{align}
\label{gbsde-ito1-Lp-} 2Y_{s}f(s,Y_s,Z_s,V_s)&
\le2\mu |Y_{s}|^2+2|Y_{s}||f_s|+2L|Y_{s}||Z_{s}|+2L|Y_{s}|
\big\|V_{s}(e)\big\|_{\cL
^2_\lambda}\nn
\\[-2pt]
&\le \biggl(\frac{1}{\varepsilon}+2\bigl(\mu+L^2\bigr)
\biggr)|Y_{s}|+2|Y_{s}|f_s+
\frac{1}{2}|Z_{s}|^2+\varepsilon
\big\|V_{s}(e)\big\| ^2_{\cL^2_\lambda}\nn
\\[-2pt]
&\le\frac{1}{\varepsilon}|Y_{s}|^2+2|Y_{s}|f_s+
\frac
{1}{2}|Z_{s}|^2+\varepsilon
\big\|V_{s}(e)\big\|^2_{\cL^2_\lambda}.
\end{align}
Thus, since $\t_n\le T$, we deduce that
{\allowdisplaybreaks
\begin{align*}
&\frac{1}2\int_{0}^{\t_n}|Z_s|^2 ds+\int_{0}^{\t_n}\int_U\big|V_{s}(e)\big|^2\pi(de,ds)+\int_{0}^{\t_n}d[M]_s\\[-1pt]
&\quad \le\Sup_{t\in[0,T]}|Y_t|^2+\frac{1}{\varepsilon}\int_0^T|Y_{s}|^2 ds+2\Sup_{t\in[0,T]}|Y_t|\int_{0}^{T}f_s ds\\
&\qquad +2\Sup_{t\in[0,T]}|Y_t|\int_{0}^{T}d|R|_s+\varepsilon\int_0^{\t_n}\int_U\big|V_{s}(e)\big|^2\lambda(de) ds+2\Biggl\llvert \int_0^{\t_n}Y_{s}Z_s dW_s\Biggr\rrvert\\
&\qquad +2\Biggl\llvert \int_0^{\t_n}\int_UY_{s-}V_{s}(e)^2\widehat{\pi }(de,ds)\Biggr\rrvert +2\Biggl\llvert \int_0^{\t_n}Y_{s-}dM_s\Biggr\rrvert\\
&\quad \le \biggl(3+\frac{T}{\epsilon} \biggr)\Sup_{t\in[0,T]}|Y_t|^2+\Biggl(\int_{0}^{T}f_s ds\Biggr)^2+ \Biggl(\int_{0}^{T}d|R|_s\Biggr)^2\\
&\qquad +\varepsilon\int_0^{\t_n}\int_U\big|V_{s}(e)\big|^2 \lambda(de) ds+\Biggl\llvert \int_0^{\t_n} Y_{s}Z_s dW_s\Biggr\rrvert\\
&\qquad +\Biggl\llvert \int_0^{\t_n}\int_UY_{s-}V_{s}(e)^2\widehat{\pi }(de,ds)\Biggr\rrvert +\Biggl\llvert \int_0^{\t_n}Y_{s-}dM_s\Biggr\rrvert.
\end{align*}}%
It follows that there exists a constant $c_p>0$, depending only on $p$,
such that
\begin{align}
\label{gbsde-ito2-Lp-pgreaterthan2}
& \Biggl(\int_0^{\t_n}|Z_s|^2 ds \Biggr)^\frac{p} {2}+ \Biggl(\int_0^{\t_n}\int_U\big|V_{s}(e)\big|^2 \pi(de,ds)\Biggr)^\frac{p} {2} + \Biggl(\int_0^{\t_n}d[M]_s\Biggr)^\frac{p} {2}\nn\\
&\quad \le c_p \Biggl[ \biggl(3+\frac{T}{\epsilon} \biggr)^\frac{p}{2}\Sup _{t\in[0,T]}|Y_{t}|^p+ \Biggl(\int_{0}^Tf_s ds \Biggr)^p+ \Biggl(\int_0^{\t_n}d|R|_s\Biggr)^p\nn\\
&\qquad +\epsilon^\frac{p} {2} \Biggl(\int_0^{\t_n}\int_U\big|V_{s}(e)\big|^2 \lambda(de) ds\Biggr)^\frac{p} {2}+\Biggl\llvert \int_0^{\t_n}Y_{s}Z_s dW_s\Biggr\rrvert ^\frac{p} {2}\nn\\
&\qquad +\Biggl\llvert \int_0^{\t_n}\int_UY_{s-}V_{s}(e)^2\widehat {\pi}(de,ds)\Biggr\rrvert ^\frac{p} {2}+\Biggl\llvert \int_0^{\t_n}Y_{s-} dM_s\Biggr\rrvert ^\frac{p} {2} \Biggr].
\end{align}
Since $\frac{p}{2}\ge1$, we can apply the BDG inequality to obtain
\begin{align}
c_p\E\Biggl\llvert \int_0^{\t_n}Y_{s}Z_s dW_s\Biggr\rrvert ^\frac{p} {2}&\le d_p\E\Biggl[ \Biggl(\int_0^{\t_n}|Y_{s}|^2|Z_{s}|^2 ds \Biggr)^\frac {p} {4} \Biggr]\nn\\
&\le\frac{d^2_p}{4}\E\Bigl(\sup_{t\in[0,T]}|Y_{t}|^p\Bigr)+\frac{1}{2}\E \Biggl(\int_0^{\t_n}|Z_s|^2 ds \Biggr)^\frac{p} {2},\label{BDG-Brownian-p-greater-than-2}\\[9pt]
c_p\E\Biggl\llvert \int_0^{\t_n}Y_{s-}dM_s\Biggr\rrvert ^\frac{p} {2}&\le d_p\E \Biggl[ \Biggl(\int_0^{\t_n}|Y_{s-}|^2d[M]_{s}\Biggr)^\frac{p} {4} \Biggr]\nn\\
&\le\frac{d^2_p}{4}\E\Bigl(\sup_{t\in[0,T]}|Y_{t}|^p\Bigr)+\frac{1}{2}\E [M]_{\t_n}^\frac{p} {2},
\end{align}
and
{\allowdisplaybreaks
\begin{align}
\label{BDG-poisson-p-greater-than-2}
&c_p\E\Biggl\llvert \int_0^{T}\int_UY_{s-}V_s(e)\widehat\pi(de,ds)\Biggr\rrvert ^\frac{p} {2}\nn\\
&\quad \le d_p\E \Biggl[ \Biggl(\int_0^{\t_n}|Y_{s}|^2\big|V_{s}(e)\big|^2\pi (de,ds) \Biggr)^\frac{p} {4} \Biggr]\nn\\
&\quad \le\frac{d^2_p}{4}\E\Bigl(\sup_{t\in[0,T]}|Y_{t}|^p\Bigr) +\frac{1}{2}\E \Biggl(\int_0^{\t_n}\big|V_{s}(e)\big|^2\pi(de,ds) \Biggr)^\frac{p} {2}.
\end{align}}%
Plugging estimates \eqref{BDG-Brownian-p-greater-than-2}--\eqref
{BDG-poisson-p-greater-than-2} into \eqref
{gbsde-ito2-Lp-pgreaterthan2} and then taking the expectation, we get
\begin{align}
&\frac{1}{2}\E \Biggl[ \Biggl(\int_0^{\t_n}|Z_s|^2 ds \Biggr)^\frac {p} {2}+ \Biggl(\int_0^{\t_n}\int_U\big|V_{s}(e)\big|^2 \pi(de,ds)\Biggr)^\frac{p} {2} + \Biggl(\int_0^{\t_n}d[M]_s\Biggr)^\frac{p} {2} \Biggr]\nn\\
&\quad \le C(p,T,\epsilon)\E\Sup_{t\in[0,T]}|Y_{t}|^p+c_p\E \Biggl(\int_{0}^Tf_s ds\Biggr)^p+c_p\E \Biggl(\int_0^{\t_n}d|R|_s\Biggr)^p\nn\\
&\qquad +c_p\epsilon^\frac{p} {2}\E \Biggl[ \Biggl(\int_0^{\t_n}\int_U\big|V_{s}(e)\big|^2\lambda(de) ds \Biggr)^\frac{p} {2} \Biggr] .\nn
\end{align}
By \cite[Section~4]{Lenglartetal79} or \cite[Lemma 2.1]{Dzpharidze90}
we have that, for some constant $\eta_p>0$,
%
\begin{equation}
\E \Biggl(\int_0^{\t_n}\int_U\big|V_{s}(e)\big|^2
\lambda(de) ds \Biggr)^\frac{p} {2} \le\eta_p\E \Biggl(\int
_0^{\t_n}\int_U\big|V_{s}(e)\big|^2
\pi (de,ds) \Biggr)^\frac{p} {2}.
\end{equation}
Thus, choosing $\epsilon$ small enough and depending only on $p$,
we deduce that 
\begin{align*}
&\E \Biggl[ \Biggl(\int_0^{\t_n}|Z_s|^2 ds \Biggr)^\frac {p} {2}+ \Biggl(\int_0^{\t_n}\int_U\big|V_{s}(e)\big|^2 \lambda(de)ds\Biggr)^\frac{p} {2} + \Biggl(\int_0^{\t_n}d[M]_s\Biggr)^\frac{p} {2} \Biggr]\\
&\quad \le \widetilde C(p,T)\E\Sup_{t\in[0,T]}|Y_{t}|^p+\widetilde {C}_p\E \Biggl[ \Biggl(\int_{0}^Tf_s ds \Biggr)^p + \Biggl(\int_0^{T}d|R|_s\Biggr)^p \Biggr].
\end{align*}
Finally, letting $n$ to $+\infty$ and using Fatou's lemma, \eqref
{gbsde-ito2-Lp-pgreaterthan2-T} follows.

{\bf Step 2.} Since $p\ge2$, we can apply It\^o's formula 
with the $\cC^2$ function $|y|^p$ to $|Y_t|^p$. Note that
\[
\frac{\partial\theta}{\partial y_i}(y)=py_i|y|^{p-2},\qquad
\frac
{\partial^2 \theta}{\partial y_i \partial y_j}(y)=p|y|^{p-2}\delta _{i,j}+p(p-2)y_iy_j|y|^{p-4},
\]
where $\delta_{i,j}$ is the Kronecker delta. Thus, for every $t\in
[0,T]$, we have
\begin{align}
| Y_{t}|^p &= |\xi|^p+p\int_{t}^{T}Y_{s-}|Y_{s-}|^{p-2}dR_s+p\int_{t}^{T} Y_s | Y_s|^{p-2}f(s,Y_s,Z_s,V_s)ds\nonumber\\
& \quad- p \int_{t}^{T} Y_{s-}|Y_{s-}|^{p-2} dM_s - p \int_{t}^{T} Y_{s}|Y_{s}|^{p-2}Z_s dW_s \nn\\
& \quad- p \int_{t}^{T} \!\int_U\bigl( Y_{s-} |Y_{s-}|^{p-2} V_s(e)\bigr) \widehat{\pi}(de,ds) \,{-}\, \frac{1}{2}\int_t^T \!\!\operatorname{Trace}\bigl(D^2\theta(Y_s)Z_sZ_s^t\bigr) ds\nn\\
& \quad- \int_{t}^{T} \!\int_U\bigl( \big|Y_{s-} + V_s(e) \big|^p -|Y_{s-}|^p - pY_{s-} |Y_{s-}|^{p-2}V_s(e) \bigr) \pi(de,ds)-\aleph_t,\label{ito-formula-cauchy-p-greater-than-2}
\end{align}
where
\begingroup
\abovedisplayskip=7.2pt
\belowdisplayskip=7.2pt
\begin{align}
\aleph_t&=\frac{1}{2} \int_t^T
\sum_{1\le i,j\le d}\frac{\partial
^2 \theta}{\partial y_i \partial y_j}(Y_s)d
\bigl[M^i,M^j\bigr]^c_s\nn
\\
&\quad +\sum_{t< s \leq T} \bigl( |Y_{s-}
+ \Delta M_s |^p - |Y_{s-}|^p -
pY_{s-} |Y_{s-}|^{p-2}\Delta M_s
\bigr).\nn
\end{align}
Following arguments from \cite[Prop.~2]{KP14}, we have that
\begin{eqnarray}
&\displaystyle \operatorname{Trace}\bigl(D^2\theta(y)zz^t\bigr)\ge p|y|^{p-2}|z|^2 ,\label{estimate-traceZ-Lp-geq-2}&\\
&\displaystyle \aleph_t \ge\alpha_p\int_t^T|Y_{s-}|d[M]_s,\label{estimate-martingale-part-Lp-geq-2}&
\end{eqnarray}
and
\begin{align}
\label{estimate-poisson-part-Lp-geq-2}
&- \int_{t}^{T} \int_U \bigl( \big|Y_{s-} + V_s(e)\big|^p - |Y_{s-}|^p - pY_{s-}|Y_{s-}|^{p-2} V_s(e) \bigr) \pi(de,ds)\nn\\
&\quad \le-p(p-1)3^{1-p}\int_t^T|Y_{s-}|^{p-2}\big|V_s(e)\big|^2\pi(de,ds),
\end{align}
where $\alpha_p=\min(\frac{p}{2},p(p-1)3^{1-p})$.

Consequently, in view of estimates \eqref{estimate-traceZ-Lp-geq-2},
\eqref{estimate-martingale-part-Lp-geq-2}, and \eqref
{estimate-poisson-part-Lp-geq-2}, Eq.~\eqref
{ito-formula-cauchy-p-greater-than-2} becomes
\begin{align}
&| Y_{t}|^p +\alpha_p\int_t^T | Y_s|^{p-2}|Z|^2_s ds+\alpha_p\int_{t}^{T}\int_U|Y_s|^{p-2}\big|V_s(e)\big|^2\pi(de,ds)\nn\\
&\qquad +\alpha_p \int_t^T|Y_s|^{p-2}d[M]^c_s+\alpha_p\sum_{t< s \leq T}|Y_s|^{p-2}|\Delta M_s|^2\nn\\
&\quad \le |\xi|^p+p\int_{t}^{T}Y_{s-}|Y_{s-}|^{p-2}dR_s+p\int_{t}^{T} Y_s | Y_s|^{p-2}f(s,Y_s,Z_s,V_s)ds\nn\\
&\qquad - p \int_{t}^{T} Y_{s-}|Y_{s-}|^{p-2} dM_s- p \int_{t}^{T}Y_{s}|Y_{s}|^{p-2} Z_s dW_s\nn\\
&\qquad - p \int_{t}^{T} \int_U Y_{s-} |Y_{s-}|^{p-2} V_s(e)\,\widehat{\pi }(de,ds).
\end{align}
But from assumption $(A)$ and the fact that $\mu\le-2L^2\le0$ (since
$\mu+2L^2\le0$) we deduce by using the inequality $ab\le\frac
{1}{2\epsilon}a^2+\frac{\epsilon}{2}b^2$ that
%
\begin{eqnarray}
\label{estimate-Yf-Lp-p>2} Y_{s}f(s,Y_s,Z_s,V_s)
\le\frac{L^2}{\epsilon}|Y_{s}|^2+|Y_{s}|f_s+
\frac{\epsilon
}{2}|Z_{s}|^2+\frac{\epsilon}{2}\int
_U\big|V_{s}(e)\big|^2 \lambda(de).
\end{eqnarray}
Choosing $\epsilon=\frac{\alpha_p}{p}$, we obtain in view of the
last inequality that
{\allowdisplaybreaks
\begin{align}
&| Y_{t}|^p +\frac{\alpha_p}{2}\int_t^T | Y_s|^{p-2}|Z|^2_s ds+\alpha_p\int_{t}^{T}\int_U|Y_s|^{p-2}\big|V_s(e)\big|^2\pi(de,ds)\nn\\
&\qquad +\alpha_p \int_t^T|Y_s|^{p-2}d[M]_s\nn\\
&\quad \le |\xi|^p+\frac{pL^2}{\alpha_p}\int_{t}^{T}|Y_s|^{p} ds+p\int_{t}^{T} | Y_s|^{p-1}f_s ds+p\int_{t}^{T} |Y_{s-}|^{p-1}dR_s\nn\\
&\qquad +\frac{\alpha_p}{2}\int_{t}^{T}|Y_s|^{p-2}\big\|V_s(e)\big\|^2_{\cL^2_{\lambda}} ds-p \int_{t}^{T} Y_{s}|Y_{s}|^{p-2} Z_s dW_s\nn\\
&\qquad - p \int_{t}^{T} \int_U Y_{s-} |Y_{s-}|^{p-2} V_s(e)\,\widehat{\pi}(de,ds) - p \int_{t}^{T}Y_{s-} |Y_{s-}|^{p-2} dM_s.\label{ito-formula-cauchy-p-greater-than-2-ineq2}
\end{align}}%
\endgroup
Let us set $X=|\xi|^p+\frac{pL^2}{\alpha_p}\int_{t}^{T}|Y_s|^{p}\,
ds+p\int_{t}^{T} | Y_s|^{p-1}|f_s| ds+p\int_{t}^{T}
|Y_{s-}|^{p-1}dR_s$, $\varUpsilon_t=\int_{0}^{t} Y_{s}|Y_{s}|^{p-2} Z_s
dW_s$, $\varTheta_t=\int_{0}^{t} \int_U Y_{s-} |Y_{s-}|^{p-2} V_s(e)\,
\widehat{\pi}(de,ds)$, and
$\varGamma_t=\break \int_{0}^{t} Y_{s-} |Y_{s-}|^{p-2} dM_s$.

It follows from the BDG inequality that 
$\varUpsilon_t$, $\varTheta_t$, and $\varGamma_t$ are uniformly integrable
martingales. Indeed, by Young's inequality we have
\begin{align}
\E\bigl([\varUpsilon]_T^\frac{1}2\bigr)&\le\E\Biggl[\Sup_{t\in[0,T]}|Y_{t}|^{p-1} \Biggl(\int_0^{T}|Z_s|^2 ds\Biggr)^\frac {1} {2} \Biggr]\nn\\
&\le\frac{p-1}{p}\E\Bigl(\Sup_{t\in[0,T]}|Y_{t}|^p\Bigr)+\frac{1}{p}\E \Biggl(\int_0^{T}|Z_s|^2 ds \Biggr)^\frac{p} {2},\label{Lp-martingale-uniformly-integ-BDG-Young}\\[6pt]
\E\bigl([\varTheta]_T^\frac{1}2\bigr)&\le\E\Biggl[\Sup_{t\in[0,T]}|Y_{t}|^{p-1} \Biggl(\int_0^{T}\big|V_s(e)\big|^2\pi(de,ds) \Biggr)^\frac{1} {2} \Biggr]\nn\\
&\le\frac{p-1}{p}\E\Bigl(\Sup_{t\in[0,T]}|Y_{t}|^p\Bigr)+\frac{1}{p}\E \Biggl(\int_0^{T}\big\|V_s(e)\big\|^2_{\cL^2_\lambda}ds \Biggr)^\frac{p}{2},\label{Lp-martingale-uniformly-integ-BDG-Young-2}
\end{align}
and
%
\begin{equation}
\label{Lp-martingale-uniformly-integ-BDG-Young-3} \E\bigl([\varGamma]_T^\frac{1}2\bigr)\le\E
\Bigl[\Sup_{t\in
[0,T]}|Y_{t}|^{p-1}[M]_T^\frac{1}
{2} \Bigr] \le\frac{p-1}{p}\E\Bigl(\Sup_{t\in[0,T]}|Y_{t}|^p
\Bigr)+\frac{1}{p}\E [M]_T^\frac{p} {2}.
\end{equation}
The claim holds since the last terms of \eqref
{Lp-martingale-uniformly-integ-BDG-Young}, \eqref
{Lp-martingale-uniformly-integ-BDG-Young-2}, and \eqref
{Lp-martingale-uniformly-integ-BDG-Young-3} are finite. This is due to
the fact that $Y\in\cS^p$, which implies by the first step of the
proof that $Z\in\cM^p$, $V\in\cL^p$,
and $M\in\M^p$.
Moreover, we have
%
\begin{equation}
\label{Lp-2-equa-Y-pi-lambda} \E\int_{t}^{T} \int
_U|Y_s|^{p-2}\big|V_s(e)\big|^2
\pi(de,ds)=\E\int_{t}^{T}|Y_s|^{p-2}
\big\|V_s(e)\big\|^2_{\cL^2_{\lambda}} ds.
\end{equation}
Hence, in view of \eqref{Lp-2-equa-Y-pi-lambda}, taking the
expectation in \eqref{ito-formula-cauchy-p-greater-than-2-ineq2} yields
\begin{align}
\label{ineq3.40}
&\frac{\alpha_p}{2}\E\int_t^T |Y_s|^{p-2} |Z|^2_s ds+\frac{\alpha_p}{2}\E\int_{t}^{T}|Y_s|^{p-2}\big\|V_s(e)\big\|^2_{\cL^2_{\lambda}} ds\nn\\
&\qquad +\alpha_p \E\int_t^T |Y_s|^{p-2}d[M]_s\nn\\
&\quad \le\E[X].
\end{align}

Furthermore, coming back to \eqref
{ito-formula-cauchy-p-greater-than-2-ineq2}, we deduce in view of
\eqref{ineq3.40} that\vadjust{\eject}
\begin{align}
\E\Sup_{s\in[t,T]}|Y_{s}|^p &\le2\E X+p\E \Biggl(\sup_{s\in[t,T]}\Biggl\llvert \int_s^T\int_UY_{u^-}|Y_{u^-}|^{p-2} V_u(e)\widehat {\pi}(de,du)\Biggr\rrvert \Biggr)\nn\\[-1pt]
&\quad +p\E \Biggl(\sup_{s\in[t,T]}\Biggl\llvert \int_s^TY_{u}|Y_{u}|^{p-2}Z_u dW_u\Biggr\rrvert \Biggr)\nn\\[-1pt]
&\quad +p\E \Biggl(\sup_{s\in[t,T]}\Biggl\llvert \int_s^TY_{u^-}|Y_{u^-}|^{p-2} dM_u\Biggr\rrvert \Biggr).
\end{align}
The BDG inequality implies that
\begin{align}
&p\E \Biggl(\sup_{s\in[t,T]}\Biggl\llvert \int_s^TY_{u}|Y_{u}|^{p-2}Z_u dW_u\Biggr\rrvert \Biggr)\nn\\[-1pt]
&\quad \le d_p \E \Biggl[ \Biggl(\int_t^T\,|Y_{s}|^{2p-2}|Z_s|^2ds\Biggr)^\frac{1}2 \Biggr]\nn\\[-1pt]
&\quad \le\frac{1}4\E\Sup_{s\in[t,T]}|Y_{s}|^p+4d^2_p\E \Biggl(\int_t^T\, |Y_{s}|^{p-2}|Z_s|^2du\Biggr),\label{BDG-step2-BM-Lp}\\
&p\E \Biggl(\sup_{s\in[t,T]}\Biggl\llvert \int_s^T\int_UY_{u^-}|Y_{u^-}|^{p-2} V_u(e)\widehat{\pi}(de,du)\Biggr\rrvert \Biggr)\nn\\[-1pt]
&\quad \le d_p\E \Biggl(\int_t^T\int_U|Y_{u}|^{2p-2}\big|V_s(e)\big|^2\pi (de,ds) \Biggr)^\frac{1}2\nn\\[-1pt]
&\quad \le\frac{1}4\E\Sup_{s\in[t,T]}|Y_{s}|^p+4d^2_p\E \Biggl(\int_t^T\, |Y_{u}|^{p-2}\big\|V(e)_u\big\|^2_{\cL_\lambda^2}du \Biggr),\label{BDG-step2-Poisson-Lp}
\end{align}
and
\begin{align}
&p\E \Biggl(\sup_{s\in[t,T]}\Biggl\llvert \int_s^TY_{u^-} |Y_{u^-}|^{p-2}dM_u\Biggr\rrvert \Biggr)\nn\\[-2pt]
&\quad \le d_p\E \bigl(|Y_{s}|^{2p-2}d[M]_s\bigr)^\frac{1}2\nn\\[-2pt]
&\quad \le\frac{1}4\E\Sup_{s\in[t,T]}|Y_{s}|^p+4d^2_p\E \Biggl(\int_t^T\, |Y_{s}|^{p-2}d[M]_s\Biggr).\label{BDG-step2-cadlag-Mart-Lp}
\end{align}
Thus, combining estimates \eqref{BDG-step2-BM-Lp}--\eqref
{BDG-step2-cadlag-Mart-Lp} with \eqref{ineq3.40}, we deduce that
%
\begin{equation}
\label{ineq3.47} \E\Sup_{s\in[t,T]}|Y_{s}|^p\le
C_p \E[X].
\end{equation}
But, applying Young's inequality, we get
{\allowdisplaybreaks
\begin{align}
pC_p\E\int_t^T\,|Y_s|^{p-1}|f_s|
 ds&\le pC_p\E \Biggl(\Sup_{s\in
[t,T]}|Y_{t}|^{p-1}
\int_t^Tf_s ds \Biggr)\nn
\\[-2pt]
&\le\frac{1}6\E\Sup_{s\in[t,T]}|Y_{t}|^{p}+\xch{d'_p
\E \Biggl(\int_t^Tf_s ds
\Biggr)^p\nn,}{d'_p
\E \Biggl(\int_t^Tf_s ds
\Biggr)^p\nn}
\end{align}}%
and
\begingroup
\abovedisplayskip=4pt
\belowdisplayskip=4pt
\begin{equation}
\nn pC_p\E\int_t^T
\,|Y_{s-}|^{p-1} dR_s\le\frac{1}6E
\Sup_{s\in
[t,T]}|Y_{t}|^{p}+d''_p
\E \Biggl(\int_t^{T}d|R|_s
\Biggr)^p.
\end{equation}
Consequently, rearranging \eqref{ineq3.47} in view of the two last
estimates implies
\begin{align}
\E\Sup_{s\in[t,T]}|Y_{s}|^p&\le C'_p \E \Biggl[|\xi|^p+ \Biggl(\int_t^Tf_s ds \Biggr)^p+\Biggl(\int_t^{T}d|R|_s\Biggr)^p \Biggr]\nn\\[-2pt]
&\quad +C''_p\int_t^T\E\Sup_{u\in[s,T]}|Y_{u}|^p ds, \hspace{3mm} t\in[0,T].\nn
\end{align}
Finally, using Gronwall's lemma, we deduce that
%
\begin{equation}
\E\Sup_{t\in[0,T]}|Y_{t}|^p\le C'_pe^{C''_pT}
\E \Biggl[|\xi |^p+ \Biggl(\int_0^Tf_s
 ds \Biggr)^p+ \Biggl(\int_0^{T}d|R|_s
\Biggr)^p \Biggr].
\end{equation}
\endgroup
This, combined with \eqref{gbsde-ito2-Lp-pgreaterthan2-T}, ends the proof.
\end{proof}
\begin{lemma}\label{Lp-estimates-variation}
Let $(\xi,f,R)$ and $(\xi',f',R')$ be two sets of data, each
satisfying assumptions $(H1)$--$(H6)$.
Let $(Y,Z,V,M)$ and $(Y',Z',V',M')$ denote respectively an $\L
^p$-solution of GRBSDE \eqref{RBSDE-jumps} with data $(\xi,f,R)$ and
$(\xi',f',R')$.
Define
\begingroup
\abovedisplayskip=4pt
\belowdisplayskip=4pt
\[
(\bar{Y},\bar{Z},\bar{V},\bar{M},\bar{\xi},\bar{f},\bar {R})=
\bigl(Y-Y',Z-Z',V-V',M-M',
\xi-\xi',f-f',R-R'\bigr).
\]
Then there exists a constant $C>0$, depending on $p$ and $T$, such
that, for every $a\ge\mu+2L^2$,
\begin{align}
\label{ineq2.63}
&\E \Biggl[\Sup_{t\in[0,T]}e^{apt}|\bar Y_t|^p+ \Biggl(\int_0^Te^{2at}|\bar Z_t|^2dt \Biggr)^\frac{p} {2}\nn\\[-2pt]
&\qquad + \Biggl(\int_{0}^{T}\int_Ue^{2at}|\bar V_{t}|^2\lambda(de)dt \Biggr)^\frac{p} {2}+e^{apT}[\bar M]_T^\frac{p} {2}
\Biggr]\nn
\\[-2pt]
&\quad {\le}\, C\E \Biggl[e^{apT}|\bar\xi|^p+ \Biggl(\int
_0^Te^{at}\big|\bar f\bigl(s,Y'_t,Z'_t,V'_t
\bigr)\big| dt \Biggr)^p+ \Biggl(\int_0^Te^{at}d|
\bar R|_t \Biggr)^p \Biggr].
\end{align}
\endgroup
\end{lemma}
\begin{proof}
By an already used change-of-variable argument we may assume that $a=0$.
Obviously, $(\bar{Y},\bar{Z},\bar{V},\bar{M})$ solves the following
GBSDE in $\varXi^p$:
\begingroup
\abovedisplayskip=4pt
\belowdisplayskip=4pt
\begin{align}
\label{Lp-p2-ito-formula-variation} \bar{Y}_t&=\bar\xi+\int^T_t
\bigl(f(s,Y_s,Z_s,V_s)-f'
\bigl(s,Y'_s,Z'_s,V'_s
\bigr) \bigr) ds+\int_t^Td\bar R_s
-\int^T_t\bar Z_s
 dW_s\nn
\\[-1pt]
& \quad -\int_{t}^{T}\int
_U\bar V_{s}(e)\widehat{\pi }(de,ds)-\int
_t^Td\bar M_s,\hspace{3mm} t
\in[0,T].
\end{align}
It follows from $(H3)$, $(H5)$, and $(H6)$ that
\begin{align}
&\operatorname{\widehat{sgn}}(\bar{y}) \bigl(f(t,y,z,v)-f'\bigl(t,y',z',v'
\bigr)\bigr)\nn
\\[-1pt]
&\quad =\operatorname{\widehat{sgn}}(\bar{y}) \bigl(f(t,y,z,v)-f\bigl(t,y',z',v'
\bigr) \bigr)+\operatorname{\widehat{sgn}}(\bar{y})\bar{f}\bigl(t,y',z',v'
\bigr)\nn
\\[-1pt]
&\quad =\operatorname{\widehat{sgn}}(\bar{y}) \bigl[f(t,y,z,v)-f\bigl(t,y',z,v\bigr)+f
\bigl(t,y',z,v\bigr)-f\bigl(t,y',z',v
\bigr)\nn
\\[-1pt]
&\qquad +f\bigl(t,y',z',v\bigr)-f
\bigl(t,y',z',v'\bigr) \bigr]+
\operatorname{\widehat{sgn}}(\bar {y})\bar{f}\bigl(t,y',z',v'
\bigr)\nn
\\[-1pt]
&\quad \le\big|\bar{f}\bigl(t,y',z',v'\bigr)\big|+
\mu|y|+L|z|+L\|v\|_{\cL^2_\lambda},
\end{align}
\endgroup
which means that assumption $(A)$ is satisfied for the generator of
GBSDE \eqref{Lp-p2-ito-formula-variation}, with $f_t\equiv|\bar
{f}(t,y',z',v')|$. Thus, by Lemma~\ref{Lp-estimates} the desired
estimate follows, which ends the proof of Lemma~\ref{Lp-estimates-variation}.
\end{proof}

Now we are able to give the main result of this subsection, the
existence and uniqueness of an $L^p$-solution of GBSDE \eqref
{RBSDE-jumps} in the case $p\ge2$.
\begin{thm}\label{existence-uniqueness-in-Lp-multi-case-1<p<2}
Let $p\ge2$ and assume that $(H1)$--$(H6)$ hold. Then, there exists a
unique $\L^p$-solution $(Y,Z,V,M)$
for the GBSDE \eqref{RBSDE-jumps}.
\end{thm}
\begin{proof}
Uniqueness follows immediately from Lemma~\ref
{Lp-estimates-variation}. Now we deal with the existence. Set
$T_n(x)=\frac{xn}{|x|\vee n}$ for $n\in\N^*$ and define $\xi_n$,
$f_n$, and $R^n$ as follows:
\begingroup
\abovedisplayskip=4pt
\belowdisplayskip=4pt
\begin{align*}
\xi_n&=T_n(\xi),\qquad f_n(t,y,z,v)=f(t,y,z,v)-f(t,0,0,0)+T_n\bigl(f(t,0,0,0)\bigr),\nn\\[-2pt]
R^n_t&=\int_0^t\1_{\{|R|_s\le n\}} dR_s.
\end{align*}
Let $(Y^n,Z^n,V^n,M^n)$ be a solution of GBSDE \eqref{RBSDE-jumps}
associated with $(\xi_n,f_n+dR_n)$. Hence, by Theorem~\ref
{existence-uniqueness-in-L2-multi-case}, for every $n\in\N$, there
exists a unique solution to $(Y^n,Z^n,V^n,\break M^n)\in\varXi^2$ of GBSDE
\eqref{RBSDE-jumps} associated with $(\xi^n,f_n+dR^n)$, but in fact
also in $\varXi^p$, $p\ge2$, according to Lemma~\ref{Lp-estimates}.
Our goal now is to show that $(Y^n,Z^n,V^n,M^n)$ is a Cauchy sequence
in $\varXi^p$. For $m\ge n$, applying Lemma~\ref{Lp-estimates-variation}
yields
\begin{align*}
&\E \Biggl[\Sup_{t\in[0,T]}\big|Y^{m}_t-Y^{n}_t\big|^2+\Biggl(\int_0^T\big|Z^{m}_t-Z^{n}_t\big|^2 dt \Biggr)^\frac{p} {2}\\[-1pt]
&\qquad + \Biggl(\int_{0}^{T}\int_U\big|V^{m}_t(e)-V^{n}_t(e)\big|^2\lambda(de) dt \Biggr)^\frac{p} {2}+ \bigl[M^m-M^n\bigr]_T^\frac{p} {2} \Biggr]\\[-2pt]
&\quad \le C\E \Biggl[|\xi_m-\xi_n|^p+ \Biggl(\int_0^T\big|T_m\bigl(f(t,0,0,0)\bigr)-T_n\bigl(f(t,0,0,0)\bigr)\big| dt \Biggr)^p\\[-2pt]
&\qquad + \Biggl(\int_0^Td\big|R^m-R^n\big|_s\Biggr)^p \Biggr].
\end{align*}
Therefore, letting $n$ and $m$ to infinity, we conclude that
$(Y^n,Z^n,V^n,M^n)$ is a Cauchy sequence in $\varXi^p$ and its limit
$(Y,Z,V,M)\in\varXi^p$ is a solution of GBSDE \eqref{RBSDE-jumps}
associated with $(\xi,f+dR)$, which ends the proof.
\end{proof}
%
\section{Comparison theorem}\label{section-CT}
In this section, we assume that $k=1$ and aim at showing a comparison
theorem for GBSDE. Our result, in particular, extends to the case of
generalized BSDEs in a general filtration the comparison theorem given in
\cite[Prop.~4]{KP14}.
We follow the argument of \cite{QS13}. In particular, we consider the
Dol\'eans--Dade exponential local martingale. Let $\alpha$, $\beta$ be
predictable processes integrable w.r.t.\ $dt$ and $dW_t$, respectively.
Let $\gamma$ be a predictable process defined on $[0, T]\times\varOmega
\times\R$ integrable w.r.t.\ $\widehat\pi(de,ds)$. For any $0\leq t
\leq s \leq T$, let $E$ be the solution of
\[
dE_{t,s} = E_{t,s-} \biggl[ \beta_s
dW_s + \int_U \gamma _s(e)
\widehat{\pi}(de,ds) \biggr],\quad E_{t,t}=1,
\]

and let $\varGamma$ be the solution of\endgroup
%
\begin{equation}
\label{eq:doleans_dade_exp} d\varGamma_{t,s} = \varGamma_{t,s-} \biggl[
\alpha_s ds + \beta_s dW_s + \int
_U \gamma_s(e)\widehat{\pi}(de,ds) \biggr],
\quad\varGamma_{t,t}=1.
\end{equation}
Of course, $\varGamma_{t,s} = \exp ( \int_t^s \alpha_r dr
) E_{t,s}$, and
\[
E_{t,s} = \exp \Biggl( \int_t^s
\beta_r dW_r- \frac{1}{2} \int
_t^s \beta^2_r dr
\Biggr) \prod_{t<r\leq s} \bigl(1+ \gamma_r(
\Delta X_r)\bigr) \xch{e^{-\gamma_r(\Delta X_r)},}{e^{-\gamma_r(\Delta X_r)}}
\]
with $X_t = \int_0^t \int_U u \pi(du,ds)$.

Note that, classically, if $\gamma_t(e)\ge-1,\ d\P\otimes ds\otimes
d\lambda(e)$-a.s., then $\varGamma_{t,.}\ge0$ a.s. (see \cite
[Prop.~3.1]{QS13}).

We make the following monotonicity assumption on $f$ w.r.t. $v$
{\setlength\leftmargini{30pt}
\begin{itemize}
\item[$(H6')$] For each $(y,z,v,v') \in\R\times\R^d \times(\cL
^2_\lambda)^2$, there exists a predictable process $\kappa= \kappa
^{y,z,v,v'} : \varOmega\times[0,T] \times U \to\R$ satisfying:
\begin{itemize}
\item[$\bullet$] $-1 \leq\kappa^{y,z,v,v'}_t(e)$;
\item[$\bullet$] $|\kappa^{y,z,v,v'}_t(e)|\le\vartheta(e)$, where
$\vartheta\in\cL^2_\lambda$ is such that
\end{itemize}
\[
f(t,y,z,v) - f\bigl(t,y,z,v'\bigr) \le\int_U
\bigl(v(e)-v'(e)\bigr) \kappa ^{y,z,v,v'}_t(e)
\xch{\lambda(de),}{\lambda(de)}
\]
$\P\otimes \mathit{Leb} \otimes\lambda$-a.e.
\end{itemize}}%
Notice that $(H6')$ implies $(H6)$ (see Section~5 in \cite{KP14}).

We begin by showing that a linear GBSDE with jumps can be written as a
conditional expectation via an exponential semimartingale. This result
will be used to prove the comparison theorem.
\begin{lemma}\label{Linear-GBSDE-Comp}
Assume that $|\beta|$ is bounded and $\alpha$ is bounded from above.
Suppose also that, $d\P\otimes dt\otimes\lambda(de)$-a.s.,
%
\begin{equation}
\label{properties of gamma-1} -1 \leq\xch{\gamma_t(e),}{\gamma_t(e)}
\end{equation}
and
%
\begin{equation}
\label{properties of gamma-2} \big|\gamma_t(e)\big|\le\vartheta(e),\ \ \text{where}\
\vartheta\in\cL ^2_\lambda.
\end{equation}
Let $(f_t)_{0\le t\le T}$ be a real-valued progressively measurable
process, and let $(Y,Z,V,\break M)$ be the solution of the linear GBSDE
\begin{align}
\label{eq:linear_BSDE}
Y_t &= \xi+ \int_t^T\biggl[ f_s + \alpha_s Y_s +\beta_s Z_s + \int_U\gamma_s(u) V_s(e) \lambda(de) \biggr] ds+\int_t^TdR_s\nn\\
&\quad - \int_t^T\int_U V_s(e) \widehat{\pi}(de,ds) -\int_t^T Z_sdW_s- \int_t^T dM_s.
\end{align}
Then, $\E\Sup_{s\in[t,T]}|\varGamma_{t,s}|^p<+\infty$, and if
%
\begin{equation}
\label{integrability-xi-f-R-comp-thm} \E \Biggl[|\xi|^p+ \Biggl(\int_0^T|f_s|
 ds \Biggr)^p+|R|_T^p \Biggr]<+\infty,
\end{equation}
then the solution $(Y,Z,V,M)$ belongs to $\varXi^p$.\vadjust{\eject}

Furthermore, the process $(Y_t)$ satisfies
%
\begin{equation}
\label{Y-cond-exp-linear-gbsde} Y_t=\E \Biggl[\xi\varGamma_{t,T}+\int
_t^T\varGamma_{t,s} f_s
ds+\int_t^T\varGamma_{t,s-}dR_s\Bigm|
\cF_t \Biggr],\quad 0\le t\le T, \ \text{a.s.}
\end{equation}
\end{lemma}
\begin{proof}
We first show that $\E\Sup_{s\in[t,T]}|\varGamma_{t,s}|^p<+\infty$.
Indeed, by \eqref{properties of gamma-1}, as mentioned previously, it
follows that $\varGamma_{t,.}\ge0$. Combining this with \eqref
{properties of gamma-2}, using the fact that $|\beta|$ is bounded and
$\alpha$ is bounded from above, and applying \cite[Prop.~A.1]{QS13}
yield that $\varGamma$ is $p$-integrable, that is, $\E|\varGamma
_{t,T}|^p<+\infty$. Hence, using Doob's inequality, we have that
%
\begin{equation}
\label{Gamma-belongs-Sp} \E\Sup_{s\in[t,T]}|\varGamma_{t,s}|^p\le
C_p \Sup_{s\in[t,T]} \E |\varGamma_{t,s}|^p
\le C_p\E|\varGamma_{t,T}|^p<+\infty,
\end{equation}
as desired.

Next, let us show that $(Y,Z,V,M)$ belongs to $\varXi^p$. Clearly, thanks
to the assumptions made on $\alpha$, $\beta$, and $\gamma$ and the
fact $(H6')$ implies $(H6)$, we can easily see that the generator of
the linear GBSDE \eqref{eq:linear_BSDE} satisfies assumptions $(H3)$,
$(H5)$, and $(H6)$. Thus, in view of this and \eqref
{integrability-xi-f-R-comp-thm}, applying Lemma \eqref{Lp-estimates}
yields the claim.

It remains to show that $(Y_t)$ \xch{satisfies}{satisfies Eq.}~\eqref{Y-cond-exp-linear-gbsde}. Indeed, by the It\^o product formula we obtain
\begin{align*}
d(Y_s\varGamma_{t,s}) & = \varGamma_{t,s-} dY_s + Y_{s-} d\varGamma_{t,s} + d [ \varGamma_{t,.} , Y ]_s                                                       \\
& = \varGamma_{t,s-} \biggl( -f_s - \alpha_s Y_s - \beta_s Z_s - \int _U \gamma_s(u) V_s(u) \lambda(de) \biggr) ds-\varGamma_{t,s-}dR_s \\
& \quad + \varGamma_{t,s-} \int_U V_s(e) \widehat{\pi}(de,ds) + \varGamma_{t,s-} Z_s dW_s+ \varGamma_{t,s-} dM_s                 \\
& \quad + Y_{s-} \varGamma_{t,s-} \biggl( \alpha_s ds + \beta_s dW_s + \int _U \gamma_s(u)\xch{\widehat{\pi}(de,ds)\biggr)}{\widehat{\pi}(de,ds)\biggr) )}          \\
& \quad + \varGamma_{t,s-} \beta_s Z_s ds + \varGamma_{t,s-} \int_U V_s(e) \gamma_s(u) \pi(de,ds)                                 \\
& = -\varGamma_{t,s} f_s ds -\varGamma_{t,s-}dR_s+\xch{dN_s,}{dN_s}
\end{align*}
with
\begin{align*}
dN_s&=\varGamma_{t,s-} \int_U\bigl(V_s(e) + Y_{s-} \gamma_s(e)+V_s(e)\gamma _s(e)\bigr) \widehat{\pi}(de,ds)\\
& \quad + \varGamma_{t,s-} (Z_s+Y_s\beta_s) dW_s+\varGamma_{t,s-}dM_s.
\end{align*}
Integrating between $t$ and $T$ yields
%
\begin{equation}
\label{linear-bsde-2} \xi\varGamma_{t,T}-Y_t=-\int
_t^T\varGamma_{t,s} f_s
ds-\int_t^T\varGamma _{t,s-}dR_s+
\int_t^TdN_s \quad\text{a.s.}
\end{equation}
In view of the boundedness assumptions made on the coefficients $\beta
$ and $\gamma$, combined with estimate \eqref{Gamma-belongs-Sp} and
the fact
that $(Y,Z,V,M)\in\varXi^p$, it follows\vadjust{\eject} that the local martingale $N$ is
a uniformly integrable martingale. Therefore, taking
the conditional expectation w.r.t. $\cF_t$ in \eqref{linear-bsde-2},
we obtain
\begin{equation}
\nn Y_t=\E \Biggl[\xi\varGamma_{t,T}+\int
_t^T\varGamma_{t,s} f_s
ds+\int_t^T\varGamma_{t,s-}dR_s\Bigm|
\cF_t \Biggr],
\end{equation}
as desired. This ends the proof.
\end{proof}
\begin{pro}\label{comp-thm}
We consider two sets of data $(\xi_1,f_1+dR_1)$ and $(\xi
_2,f_2+dR_2)$ such that $\xi_1$, $\xi_2$, $R_1$, and $R_2$ satisfy $(H1)$.
Moreover, we assume that $f_1$ and $f_2$ satisfy, respectively,
${(H1)}$--${(H6)}$ and ${(H1)}$--${(H5)}$, ${(H6')}$. Let
$(Y^1,Z^1,V^1,M^1)$ and $(Y^2,Z^2,V^2,M^2)$ be respectively solutions
of GBSDEs \eqref{RBSDE-jumps} associated with $(\xi_1,f_1+dR_1)$
and $(\xi_2,f_2+dR_2)$ in some space $\varXi^p$ with $p\ge2$. If $\xi
^1 \le\xi^2$, $f_1(t,Y^1_t,Z^1_t,\break V^1_t) \le
f_2(t,Y^1_t,Z^1_t,V^1_t)$, and for a.e. $t$, $dR^1\le dR^2$, then a.s.
$Y^1_t \le Y^2_t$ for any $t \in[0,T]$.
\end{pro}
\begin{proof}
Put
\[
\overline Y = Y^2-Y^1, \quad\overline Z =
Z^2-Z^1, \quad\overline V = V^2 -
V^1, \quad\overline M = M^2 - M^1, \quad
\overline R = R^2 - R^1.
\]
Then $(\overline Y,\overline Z,\overline V,\overline M)$ satisfies
\begin{equation*}
\overline Y_t = \overline\xi+ \int_t^T
h_s ds +\int_t^Td\overline
R_s- \int_t^T\int
_U \overline\psi_s(u) \widehat{\pi}(de,ds) -
\int_t^T \overline Z_s
dW_s- \int_t^T d\overline
M_s,
\end{equation*}
where
\[
h_s = f_2\bigl(Y^2_s,Z^2_s,
\psi^2_s\bigr) -f_1\bigl(Y^1_s,Z^1_s,
\psi^1_s\bigr).
\]
Now we define
\begin{align*}
f_s      & = f_2\bigl(Y^1_s,Z^1_s, \psi^1_s\bigr) -f_1\bigl(Y^1_s,Z^1_s, \psi^1_s\bigr),                             \\
\alpha_s & = \frac{f_2(Y^2_s,Z^1_s,\psi^1_s) - f_2(Y^1_s,Z^1_s,\psi ^1_s)}{\overline Y_s} \1_{\overline Y_s \neq0} , \\
\beta_s  & = \frac{f_2(Y^2_s,Z^2_s,\psi^1_s) - f_2(Y^2_s,Z^1_s,\psi ^1_s)}{\overline Z_s} \1_{\overline Z_s \neq0}.
\end{align*}
Then
\begin{align}
\label{generator-comp-linear-GBSDE}
h_s & = f_s + \alpha_s\overline Y_s + \beta_s \overline Z_s+f_2\bigl(Y^2_s,Z^2_s,V^2_s\bigr) - f_2\bigl(Y^2_s,Z^2_s,V^1_s\bigr) \nn\\
& \geq f_s + \alpha_s \overline Y_s + \beta_s \overline Z_s +\int_U\kappa_s^{Y^2_s,Z^2_s,V^1_s,V^2_s} \overline V_s(u) \xch{\lambda(de),}{\lambda(de)}
\end{align}
since $f_2$ satisfies $(H6')$. Moreover, since $f_2$ is Lipschitz
continuous w.r.t. $z$, $|\beta|$ is bounded
by $L$, whereas, by Assumption $(H3)$, $\alpha$ is bounded from above.
Moreover, the process $ \kappa_s^{Y^2_s,Z^2_s,V^1_s,V^2_s}$ is
controlled by $\vartheta\in\cL^2_\lambda$. Note that, since $-1
\leq\kappa^{y,z,\psi,\phi}_t(e)$, it follows that $\varGamma_{t,.}
\geq0$ a.s. Furthermore, in view of the above, we have from Lemma~\ref
{Linear-GBSDE-Comp} that $\varGamma_{t,.}\in\cS^p$.

Now applying It\^o's formula
to $\bar Y_s \varGamma_{t,s}$ for $s\in[t,T]$ and then using inequality
\eqref{generator-comp-linear-GBSDE} together with the non negativity
of $\varGamma$, we can derive, by doing the same computations
as in the proof of Lemma~\ref{Linear-GBSDE-Comp}, that
%
\begin{equation}
\label{d-y-Gamma-2} -d(\bar Y_s \varGamma_{t,s})\ge
\varGamma_{t,s} f_s ds +\varGamma_{t,s-}d\bar
R_s-dN_s,
\end{equation}
where $N$ is a local martingale. Next, applying Lemma~\ref
{Lp-estimates-variation} yields that $(\bar Y,\bar Z,\bar V,\bar M)$
belongs to $\varXi^p$. Since $\varGamma_{t,.}\in\cS^p$, this, combined
with the boundedness of $\beta$ and $\kappa^{y,z,\psi,\phi}_t(e)$,
implies that $N$ is in fact a martingale.

Therefore, integrating between $t$ and $T$ in \eqref{d-y-Gamma-2} and
then taking
the conditional expectation w.r.t. $\cF_t$, we deduce
\[
\overline Y_t \geq\E \Biggl[ \varGamma_{t,T} \overline\xi+
\int_t^T \varGamma_{t,s}
f_s ds+\int_t^T
\varGamma_{t,s-}d\overline R_s \Bigm| \cF _t \Biggr],
\quad t\in[0,T],\ \text{a.s.}
\]
To conclude, recall that $\varGamma_{t,s} \geq0$ a.s. and, by
assumptions, $\overline\xi\geq0$, $f_s \geq0$, and
$\int_t^Td\overline R_s\ge0$. Consequently, it follows that, for all
$t\in[0,T]$,\ $\overline Y_t \geq0$ a.s. Since $Y^1$ and $Y^2$ are
\cadlag processes, we obtain that
$Y^1_t\le Y^2_t$ a.s., and the conclusion follows.
\end{proof}

Notice that assumptions ${(H1)}$--${(H6)}$ made on $f_1$ are imposed
only to ensure the existence of a solution $(Y^1,Z^1,V^1,M^1)$. The
following corollary, which follows immediately from Proposition~\ref
{comp-thm}, gives again a uniqueness result for GBSDE \eqref
{RBSDE-jumps} in $\varXi^p$ in dimension $1$.
\begin{cor}
Let $p\ge2$ and assume ${(H1)}$--${(H5)}$ and $(H6')$. Then there
exists at most one solution $(Y,Z,V,M)$ to GBSDE \eqref{RBSDE-jumps}
in $\varXi^p$.
\end{cor}

\section{Generalized BSDEs with random terminal time}
In this section, we study the issue of existence and uniqueness of $\L
^p(p\ge2)$-solutions of GBSDEs with random terminal time. We follow
the approach in \cite[Section~4]{P99}. Let $\t$ be an $\cF$-stopping
time, not necessarily bounded. Assumptions considered in the case of
GBSDE with constant time (precisely, $(H2)$, $(H3)$, $(H5)$, and
$(H6)$) still hold except for $(H4)$ and $(H1)$,
for which we give the analogues for $p\ge2$:\looseness=1
{\setlength\leftmargini{30pt}
\begin{itemize}
\item[$(H4')$] $\forall r>0$, $\forall n\in\N$, the mapping $t\in
[0,T]\to\Sup_{|y|\le r}|f(t,y,0,0)-f(t,0,0,0)|$ belongs to $\L
^1(\varOmega\times(0,n))$.
\item[$(H1')$] For some $\rho\in\R$
such that $\rho>\nu:=\mu+\frac{2pL^2}{\alpha_p}$, where $\alpha
_p=\min(\frac{p}{2},p(p-1)3^{1-p})$,
%
\begin{equation}
\E \Biggl[e^{\rho p\t}|\xi|^p\,{+} \Biggl(\int_0^\t
e^{\rho\t
}\big|f(s,0,0,0)\big| ds \Biggr)^p\,{+} \Biggl(\int
_0^{\t}e^{\rho\t
}d|R|_s
\Biggr)^p \Biggr]{<}\,\infty.
\end{equation}
\end{itemize}}%
Finally, we will need the following additional assumption on $\xi$ and
$f$: \\
$(H7)$ $\xi$ is $\cF_\t$-measurable, and $\E [ (\int_0^\t e^{\rho\t}|f(t,\xi_t,\eta_t,\gamma_t)| ds )^p
]<+\infty$,
where $\xi_t=\E(\xi|\cF_t)$ and $(\eta,\gamma,N)$ are given by
the martingale representation
\[
\xi=\E(\xi)+\int_0^{+\infty}\eta_s
 dW_s+\int_0^{+\infty}\int
_U\gamma_s(e)\widehat\pi(de,ds)+\xch{N_\t,}{N_\t}
\]
with $N$ orthogonal to $W$ and $\widetilde\pi$. Moreover, the
following holds:
\[
\E \Biggl[ \Biggl(\int_0^{+\infty}|
\eta_s|^2 ds \Biggr)^\frac {p} {2}+ \Biggl(\int
_0^{+\infty}\int_U\big|
\gamma_s(e)\big|^2\,\lambda (de)ds \Biggr)^\frac{p}
{2} +[N]^\frac{p} {2}_\t \Biggr]<+\infty.
\]

Next, let us make precise the notion of a solution of GBSDE with random
terminal time.
\begin{defin}
We say that a quadruple $(Y,Z,V,M)\in\cS\times\cH(0,T)\times\cP
\times\M_{\mathit{loc}}$ with values in $\R^k\times\R^{k\times d}\times\R
^k\times\R^k$ is a solution of GBSDE \eqref{RBSDE-jumps} with random
terminal time $\t$ and data $(\xi,f+dR)$ if
\begin{itemize}
\item on $\{t\ge\t\}$, $ Y_t=\xi$ and $Z_t=V_t=M_t=0$, $\P$-a.s.,
\item$t\to f(t,Y_t,Z_t,V_t)\1_{\{t\le\t\}}\in L^1(0,+\infty)$,
$Z\in\L^2_{\mathit{loc}}(W)$, $V\in G_{\mathit{loc}}(\pi)$, and
\item$\P$-a.s., for all $t\in[0,T]$,
\begin{align}
\label{GBSDE-jumps-random-time}
Y_{t\wedge\t}&=Y_{T\wedge\t}+\int^{T\wedge\t}_{t\wedge\t}f(s,Y_s,Z_s,V_s) ds+\int_{t\wedge\t}^{T\wedge\t}dR_s -\int^{T\wedge\t}_{t\wedge\t}Z_s dW_s\nn\\
&\quad -\int_{t\wedge\t}^{T\wedge\t}\int_UV_{s}(e)\widehat\pi (de,ds)-\int_{t\wedge\t}^{T\wedge\t}dM_s.
\end{align}
\end{itemize}
Furthermore, a solution is said to be $\L^p$ if we have
\begin{align}
\label{estimate-solution-random-terminal-time-Lp-def}
&\E \Biggl[e^{p\rho(t\wedge\t)} |Y_{t\wedge\t}|^p +\int_{0}^{T\wedge\t} e^{p\rho s} |Y_{s}|^{p}ds + \int_{0}^{T\wedge\t} e^{p\rho s}|Y_{s}|^{p-2} |Z_s|^2 ds\nn\\
&\qquad + \int_{0}^{T\wedge\t} e^{p\rho s}|Y_s|^{p-2} \| V_s\|^2_{L^2_\lambda}ds + \int_{0}^{T\wedge\t}e^{p\rho s}|Y_{s}|^{p-2}d[M]_s \Biggr]\nn\\
& \quad < +\infty.
\end{align}
\end{defin}
\begin{pro}\label{prop-uniqueness-random-time}
Assume $(H1')$, $(H2)$, $(H3)$, $(H4')$, $(H5)$, $(H6)$, and $(H7)$.
Then, there exists at most one solution to GBSDE \eqref{RBSDE-jumps}
that satisfies estimate
\eqref{estimate-solution-random-terminal-time-Lp-def}.
\end{pro}
\begin{proof}
Assume that there exist two solutions $(Y,Z,V,M)$ and $(Y',Z',V',M')$
of GBSDE \eqref{GBSDE-jumps-random-time} that satisfy estimate \eqref
{estimate-solution-random-terminal-time-Lp-def}. Set $(\bar Y,\bar
Z,\bar V,\bar M)=(Y-Y',Z-Z',V-V',M-M')$.

Applying It\^o's formula, as
in step $2$ of Proposition~\ref{Lp-estimates}, to $e^{p\rho s}| \bar
{Y}_{s}|^p$ over the interval $[t \wedge\t,T\wedge\t]$, we obtain
an analogue of \eqref{ito-formula-cauchy-p-greater-than-2}
\begin{align}
\label{ito-formula-cauchy-p-greater-than-2-random-terminal-time}
&e^{p\rho(t\wedge\tau)}| \bar{Y}_{t\wedge\tau}|^p\nn \\[1pt]
&\quad  = e^{p\rho(T\wedge\tau)}| \bar{Y}_{T\wedge\tau}|^p\nn\\[1pt]
&\qquad +p\int_{t\wedge\t}^{T\wedge\t}e^{p\rho s} \bigl[ \bar{Y}_s | \bar {Y}_s|^{p-2} \bigl(f(s,Y_s,Z_s,V_s)-f \bigl(s,Y'_s,Z'_s,V'_s \bigr) \bigr)-\rho |\bar{Y}_s|^p \bigr]ds\nn\\[1pt]
&\qquad - p \int_{t\wedge\t}^{T\wedge\t}e^{p\rho s} \bar{Y}_{s-} |\bar {Y}_{s-}|^{p-2} d \bar{M}_s - p \int_{t\wedge\t}^{T\wedge\t} e^{p\rho s}\bar{Y}_{s}|\bar{Y}_{s}|^{p-2} \bar{Z}_s dW_s \nn\\[1pt]
&\qquad - p \int_{t\wedge\t}^{T\wedge\t} \int_U e^{p\rho s}\bar {Y}_{s-} |\bar{Y}_{s-}|^{p-2} \bar{V}_s(e) \widehat{\pi}(de,ds)\nn\\[1pt]
&\qquad - \frac{1}{2}\int_{t\wedge\t}^{T\wedge\t} e^{p\rho s}\operatorname{Trace}\bigl(D^2\theta(\bar{Y}_s) \bar{Z}_s\bar{Z}_s^t\bigr) ds\nn\\[1pt]
&\qquad - \int_{t\wedge\t}^{T\wedge\t} \int_U e^{p\rho s} \bigl( \big|\bar {Y}_{s-} + \bar{Y}_s(e) \big|^p \nn\\[1pt]
&\qquad - |\bar{Y}_{s-}|^p - p \bar{Y}_{s-} |\bar {Y}_{s-}|^{p-2} \bar{Y}_s(e) \bigr) \pi(de,ds)-\aleph_t,
\end{align}
where
\begin{align}
\aleph_t&=\frac{1}{2} \int_{t\wedge\t}^{T\wedge\t}e^{p\rho s}\Sum_{1\le i,j\le d}\frac{\partial^2 \theta}{\partial y_i\partial y_j}(Y_s)d\bigl[\bar{M}^i,\bar{M}^j\bigr]^c_s\nn\\
&\quad + \sum_{t\wedge\t< s \leq T\wedge\t} e^{p\rho s} \bigl( |\bar{Y}_{s-} + \Delta\bar{M}_s |^p - |\bar{Y}_{s-}|^p - pY_{s-} |\bar{Y}_{s-}|^{p-2}\Delta\bar{M}_s \bigr).
\end{align}
The following estimates, which are analogues of \eqref
{estimate-martingale-part-Lp-geq-2} and \eqref
{estimate-poisson-part-Lp-geq-2}, hold:
%
\begin{equation}
\label{estimate-martingale-part-Lp-geq-2-random-time} \aleph_t\ge\alpha_p\int
_{t\wedge\t}^{T\wedge\t}e^{p\rho s}|\bar {Y}_{s-}|\xch{d[\bar{M}]_s,}{d[\bar{M}]_s}
\end{equation}
and
\begin{align}
\label{estimate-poisson-part-Lp-geq-2-random-time}
&- \int_{t\wedge\t}^{T\wedge\t} \int_U e^{p\rho s} \bigl( \big|\bar {Y}_{s-} +\bar{V}_s(e) \big|^p - |\bar{Y}_{s-}|^p- p\bar{Y}_{s-} |\bar {Y}_{s-}|^{p-2}\bar{Y}_s(e) \bigr) \pi(de,ds)\nn\\
&\quad \le-p(p-1)3^{1-p}\int_{t\wedge\t}^{T\wedge\t}e^{p\rho s}|\bar {Y}_{s-}|^{p-2}\big|\bar{Y}_s(e)\big|^2\pi(de,ds),
\end{align}
where $\alpha_p=\min(\frac{p}{2},p(p-1)3^{1-p})$.

Therefore, rearranging \eqref
{ito-formula-cauchy-p-greater-than-2-random-terminal-time}, in view of
\eqref{estimate-traceZ-Lp-geq-2}, \eqref
{estimate-martingale-part-Lp-geq-2-random-time}, and \eqref
{estimate-poisson-part-Lp-geq-2-random-time} yields
{\allowdisplaybreaks
\begin{align}
\label{ito-formula-cauchy-p-greater-than-2-random-terminal-time-2}
& e^{p\rho(t\wedge\tau)}| \bar{Y}_{t\wedge\tau}|^p + \alpha_p\int_{t\wedge\t}^{T\wedge\t}e^{p\rho s} | \bar{Y}_s|^{p-2} |\bar {Z}_s|^2 ds \nn\\
& \qquad +\alpha_p\int_{t\wedge\t}^{T\wedge\t}e^{p\rho s} \int_U|\bar {Y}_s|^{p-2}\big| \bar{V}_s(e)\big|^2\pi(de,ds)+\alpha_p\int _{t\wedge\t }^{T\wedge\t}e^{p\rho s}|\bar{Y}_{s-}|d[ \bar{M}]_s\nn\\
& \quad \le e^{p\rho(T\wedge\tau)}| \bar{Y}_{T\wedge\tau}|^p\nn\\
& \qquad +p\int_{t\wedge\t}^{T\wedge\t}e^{\rho s} \bigl[\bar {Y}_s | \bar{Y}_s|^{p-2} \bigl(f(s,Y_s,Z_s,V_s)-f \bigl(s,Y'_s,Z'_s,V'_s \bigr) \bigr)-\rho|\bar{Y}_s|^p \bigr]ds \nn\\
& \qquad - p \int_{t\wedge\t}^{T\wedge\t}e^{\rho s} \bar {Y}_{s-} |\bar{Y}_{s-}|^{p-2} d \bar{M}_s - p \int_{t\wedge\t }^{T\wedge\t}e^{\rho s} \bar{Y}_{s}|\bar{Y}_{s}|^{p-2} \bar{Z}_s dW_s\nn \\
& \qquad - p \int_{t\wedge\t}^{T\wedge\t} \int _U e^{\rho s} \bar{Y}_{s-} | \bar{Y}_{s-}|^{p-2} \bar{V}_s(e)\,\widehat{ \pi}(de,ds).
\end{align}}%
But from the assumptions on $f$ (using \eqref{estimate-Yf-Lp-p>2} with
$\epsilon=\frac{\alpha_p}{p}$) and Young's inequality we have that
\begin{align}
&\bar{Y}_s | \bar{Y}_s|^{p-2}\bigl(f(s,Y_s,Z_s,V_s)-f\bigl(s,Y'_s,Z'_s,V'_s\bigr) \bigr)-\rho|\bar{Y}_s|^p\nn\\
&\quad \le\biggl(\mu+\frac{2pL^2}{\alpha_p}-\rho\biggr)|\bar{Y}_s|^p+\frac{\alpha_p}{p}|\bar{Y}_s|^{p-2}|\bar{Z}_s|^2+\frac{\alpha_p}{p}|\bar{Y}_s|^{p-2}\|\bar{V}_s\|^2_{\L^2_\lambda}\nn\\
&\quad \le\frac{\alpha_p}{p}|\bar{Y}_s|^{p-2}|\bar{Z}_s|^2+\frac{\alpha_p}{p}|\bar{Y}_s|^{p-2}\|\bar{V}_s\|^2_{\L^2_\lambda}.
\end{align}
Furthermore, observe that by the integrability conditions on the
solution all the local martingales appearing in
\eqref{ito-formula-cauchy-p-greater-than-2-random-terminal-time-2} are
uniformly integrable. Moreover, the following holds:
\begin{align}
\label{expect-equali-pi-lambda}
&\E\int_{t\wedge\t}^{T\wedge\t}e^{p\rho s}\int_U|Y_s|^{p-2}\big|V_s(e)\big|^2\pi(de,ds)\nn\\
&\quad =\E\int_{t\wedge\t}^{T\wedge\t}e^{p\rho s} \int_U|Y_s|^{p-2}\big|V_s(e)\big|^2\lambda(de) ds.
\end{align}
Thus, taking the expectation in \eqref
{ito-formula-cauchy-p-greater-than-2-random-terminal-time-2}, we
obtain, in view of the above, that
%
\begin{equation}
\E e^{p\rho(t\wedge\tau)}| \bar{Y}_{t\wedge\tau}|^p\le\E
e^{p\rho(T\wedge\tau)}| \bar{Y}_{T\wedge\tau}|^p.
\end{equation}
Note that the same result holds with $\rho$ replaced by $\rho'$, with
$\mu+\frac{p^2L^2}{\alpha_p^2}<\rho'<\rho$. Therefore, we have,
for any $0\le t\le T$,
%
\begin{equation}
\E e^{p\rho' (t\wedge\tau)}| \bar{Y}_{t\wedge\tau}|^p\le e^{p(\rho'-\rho)T}
\E e^{p\rho(T\wedge\tau)}| \bar{Y}_{T\wedge
\tau}|^p.
\end{equation}
Consequently, letting $T\rightarrow+\infty$, we deduce in view of
estimate \eqref{estimate-solution-random-terminal-time-Lp-def} that
$\bar{Y}_t=0$.

Since $(Y,Z,V,M)$ and $(Y',Z',V',M')$ satisfy GBSDE \eqref
{GBSDE-jumps-random-time} with $Y=Y'$, then by the uniqueness of the
Doob--Meyer decomposition of semimartingales it follows that
$(Z,V,M)=(Z',V',M')$, whence the uniqueness of the solution of \eqref
{GBSDE-jumps-random-time}.
\end{proof}
\begin{pro}
Assume that $(H1')$, $(H2)$, $(H3)$, $(H4')$, $(H5)$, $(H6)$, and
$(H7)$ are in force. Then, GBSDE \eqref{RBSDE-jumps} has a solution satisfying
{\allowdisplaybreaks
\begin{align}
\label{estimate-solution-random-terminal-time}
&\E \Biggl[\sup_{t\ge0}e^{p\rho(t\wedge\t)}|Y_{t\wedge\tau}|^p+e^{p\rho(t\wedge\t)} |Y_{t\wedge\t}|^p+ \int_{0}^{T\wedge\t} e^{p\rho s}|Y_{s}|^{p} ds\nn\\
&\qquad + \int_{0}^{T\wedge\t} e^{p\rho s}|Y_{s}|^{p-2} |Z_s|^2 ds+ \int_{0}^{T\wedge\t} e^{p\rho s} |Y_s|^{p-2}\|V_s\| ^2_{L^2_\lambda} ds\nn\\
&\qquad + \int_{0}^{T\wedge\t} e^{p\rho s}|Y_{s-}|^{p-2} d[ M]_s \Biggr]\nn\\
&\quad \le C\E \Biggl[e^{ p\rho\t}|\xi|^p+ \Biggl(\int_0^\t e^{\rho\t}\big|f(s,0,0,0)\big| ds\Biggr)^p+ \Biggl(\int_0^{\t}e^{\rho\t}d|R|_s\Biggr)^p \Biggr].
\end{align}}%
Moreover,
\begin{align}
\label{estimate-solution-random-terminal-time-2}
&\E \Biggl[ \Biggl(\int_{0}^{\t}e^{2\rho s}|Z_s|^2 ds \Biggr)^\frac {p} {2}+ \Biggl(\int_{0}^{\t}\int_Ue^{2\rho s}\big| V_{s}(e)\big|^2 \lambda (de)ds \Biggr)^\frac{p} {2}\nn \\
&\qquad + \Biggl(\int_{0}^{\t}e^{2\rho s}d[M]_s \Biggr)^\frac {p} {2} \Biggr]\nn                                                                                   \\
&\quad \le C\E \Biggl[e^{ p\rho\t}|\xi|^p+ \Biggl(\int _0^\t e^{\rho\t }\big|f(s,0,0,0)\big| ds \Biggr)^p+ \xch{\Biggl(\int_0^{\t}e^{\rho\t }d|R|_s \Biggr)^p \Biggr],}{\Biggl(\int_0^{\t}e^{\rho\t }d|R|_s \Biggr)^p \Biggr]}
\end{align}
for some constant $C>0$ depending only on $p$, $L$, and $\mu$.
\end{pro}
\begin{proof}
We follow the line of the argument of \cite[Thm.~4.1]{P99}. For each
$n\in\N$, we construct a solution $\{(Y^n,Z^n,V^n,M^n)\}$ as follows.
By Theorem~\ref{existence-uniqueness-in-Lp-multi-case-1<p<2}, on the
interval $[0,n]$,
\begin{align}
\label{RBSDE-jumps-random-time-approximation}
Y^n_{t} & =\E(\xi\mid \cF_n)+\int ^{n}_{t}\1_{[0,\t ]}(s)f\bigl(s,Y^n_s,Z^n_s,V^n_s \bigr) ds+\int_{t}^{n}\1_{[0,\t]}(s)dR_s \nn \\
        & \quad -\int^{n}_{t}Z^n_s  dW_s-\int_{t}^{n}\int _UV^n_{s}(e)\widehat\pi(de,ds)-\int _{t}^{n}dM^n_s,
\end{align}
and for $t\ge n$, we have by assumption $(H7)$ that
$Y_t^n=\xi_t$, $Z^n_t=\eta_t$, $V^n_t(e)=\gamma_t(e)$, $M^n_t=N_t$.

{\bf Step 1.} We first show that $(Y^n,Z^n,V^n,M^n)$ satisfies estimate
\eqref{estimate-solution-random-terminal-time}.
Applying It\^o's formula to $e^{p\rho s}| Y^n_{s}|^p$ over the interval
$[t \wedge\t,T\wedge\t]$ for $0\le t\le T\le n$ and combining
with \eqref{estimate-traceZ-Lp-geq-2}, \eqref
{estimate-martingale-part-Lp-geq-2-random-time}, and \eqref
{estimate-poisson-part-Lp-geq-2-random-time} yield
\begin{align}
& e^{p\rho(t\wedge\tau)}\big| Y^n_{t\wedge\tau}\big|^p +\alpha_p\int_{t\wedge\t}^{T\wedge\t}e^{p\rho s} \big| Y^n_s\big|^{p-2} \big|Z^n_s\big|^2 ds\nn                                                                                \\
& \qquad +\alpha_p\int_{t\wedge\t}^{T\wedge\t}e^{p\rho s} \int_U\big|Y^n_s\big|^{p-2}\big|V^n_s(e)\big|^2 \pi(de,ds)+\alpha_p\int_{t\wedge\t }^{T\wedge\t}e^{p\rho s}\big|Y^n_{s-}\big|d \bigl[M^n\bigr]_s\nn                                \\
& \quad \le e^{p\rho(T\wedge\tau)}\big| Y^n_{T\wedge\tau}\big|^p+p\int _{t\wedge \t}^{T\wedge\t}e^{\rho s} \bigl[Y^n_s \big| Y^n_s\big|^{p-2}f\bigl(s,Y^n_s,Z^n_s,V^n_s \bigr)-\rho\big|Y^n_s\big|^p \bigr]ds\nn                            \\
& \qquad +p\int_{t\wedge\t}^{T\wedge\t}e^{\rho s}Y^n_s \big| Y^n_s\big|^{p-2}dR_s- p \int _{t\wedge\t}^{T\wedge\t}e^{\rho s} Y^n_{s-} \big|Y^n_{s-}\big|^{p-2} dM^n_s \nn                                                \\
& \qquad - p \int_{t\wedge \t}^{T\wedge\t}e^{\rho s} Y^n_{s}\big|Y^n_{s}\big|^{p-2} Z^n_s dW_s \nn\\
& \qquad - p \int_{t\wedge\t}^{T\wedge\t} \int_U e^{\rho s} Y^n_{s-} \big|Y^n_{s-}\big|^{p-2} V^n_s(e) \,\widehat{\pi}(de,ds) \nn.
\end{align}
But, from the assumptions on $f$ combined with Young's inequality we
get, for a small enough constant $\delta>0$,
\begin{align}
y|y|^{p-2} f(t,y,z,v) & \leq \biggl( \mu+\frac{2pL^2}{\alpha _p-p\delta} \biggr) |y|^p +|y|^{p-1}\big|f(t,0,0,0)\big|\nn                                                                \\
                      & \quad + \biggl( \frac{\alpha_p}{p} -\delta \biggr) |y|^{p-2} |z|^2+ \biggl( \frac{\alpha_p}{p} -\delta \biggr) |y|^{p-2} \|v\|_{\L^2_\lambda}.\nn
\end{align}
Then, choosing $\delta$ such that $\mu+\frac{2pL^2}{\alpha
_p-p\delta}< \rho$, we deduce from the above that
\begin{align}
\label{random-time-estimate-n}
& e^{p\alpha(t\wedge\tau)} \big|Y^n_{t\wedge\tau}\big|^p+p \bar\rho\int_{t\wedge\tau}^{T\wedge\tau}e^{p\alpha s}\big|Y^n_{s}\big|^p  ds + p\delta \int_{t\wedge\tau}^{T\wedge\tau} e^{p\alpha s} \big|Y^n_{s}\big|^{p-2} \big|Z^n_s\big|^2 ds \nn \\
& \qquad + \alpha_p\int_{t\wedge\tau}^{T\wedge\tau}e^{p\alpha s}\big|Y^n_{s}\big|^{p-2} \big|V^n_s(e)\big|^2\pi(de,ds)+\alpha_p \int_{t\wedge\tau }^{T\wedge\tau} e^{p\alpha s} \big|Y^n_{s}\big|^{p-2} d\bigl[M^n\bigr]_s\nn                     \\
& \qquad -p \biggl(\frac{\alpha_p}{p} -\delta \biggr)\int_{t\wedge\tau}^{T\wedge \tau}e^{p\alpha s}\big|Y^n_{s}\big|^{p-2} \big\|V^n_s\big\|^2_{\cL^2_\lambda}ds \nn                                                                      \\
& \quad \leq e^{p\alpha(T\wedge\tau)}\big|Y^n_{T\wedge\tau}\big|^p + p \int _{t\wedge\tau}^{T\wedge\tau} e^{p\alpha s}\big|Y^n_s\big|^{p-1}\big|f(s,0,0,0)\big|  ds\nn                                                                               \\
& \qquad +p\int_{t\wedge\t}^{T\wedge\t}e^{\rho s} \big| Y^n_s\big|^{p-1}d|R|_s- p \int _{t\wedge\t}^{T\wedge\t}e^{\rho s} Y^n_{s-} \big|Y^n_{s-}\big|^{p-2} dM^n_s \nn                                                                    \\
& \qquad - p \int_{t\wedge\t}^{T\wedge\t}e^{\rho s} Y^n_{s}\big|Y^n_{s}\big|^{p-2} Z^n_s dW_s\nn                                                                                                                                  \\
& \qquad - p \int_{t\wedge\t}^{T\wedge\t} \int_U e^{\rho s} Y^n_{s-} \big|Y^n_{s-}\big|^{p-2} V^n_s(e)\,\widehat{\pi}(de,ds).
\end{align}
Note that all the local martingales in the last inequality are true
martingales. Thus, taking the expectation in \eqref
{random-time-estimate-n}, we get in view of \eqref
{expect-equali-pi-lambda} that
\begin{align}
\label{ineq5.11}
& \E \Biggl[e^{p\rho(t\wedge\tau)} \big|Y^n_{t\wedge\tau}\big|^p+p \delta \int_{t\wedge\tau}^{T\wedge\tau}e^{p\rho s}\big|Y^n_{s}\big|^p  ds + p\delta\int_{t\wedge\tau}^{T\wedge\tau}e^{p\rho s} \big|Y^n_{s}\big|^{p-2} \big|Z^n_s\big|^2 ds\nn  \\
& \qquad +p\delta\int_{t\wedge\tau}^{T\wedge\tau}e^{p\rho s}\big|Y^n_{s}\big|^{p-2} \big\|V^n_s\big\|_{\cL^2_\lambda} ds +\alpha_p \int _{t\wedge\tau}^{T\wedge\tau} e^{p\rho s} \big|Y^n_{s}\big|^{p-2} d\bigl[M^n\bigr]_s \Biggr]\nn \\
& \quad \leq\E [X ],
\end{align}
where
\begin{align*}
X & =e^{p\rho(T\wedge\tau)}\big|Y^n_{T\wedge\tau}\big|^p + p \int _{t\wedge \tau}^{T\wedge\tau} e^{p\rho s}\big|Y^n_s\big|^{p-1}\big|f(s,0,0,0)\big|  ds\nn \\
  & \quad +p\int_{t\wedge\tau}^{T\wedge\tau}\,e^{p\rho s}\big|Y^n_{s-}\big|^{p-1}\xch{d|R|_s.}{d|R|_s}
\end{align*}
Next, as in the proof of Lemma~\ref{Lp-estimates}, including a $\sup_{s\in[t,T]}$ in \eqref{random-time-estimate-n} and applying the BDG
inequality, we get in view of \eqref{ineq5.11} that
%
\begin{equation}
\label{eq5.10} \E \Bigl[\sup_{s\in[t,T]}e^{p\rho(s\wedge\tau)}
\big|Y^n_{s\wedge
\tau}\big|^p \Bigr]\le C_p\E[X].
\end{equation}
But by Young's inequality we have
\begingroup
\abovedisplayskip=7pt
\belowdisplayskip=7pt
\begin{align}
\label{young-ineq-random-time-Yf}
 & p\E\int_{t\wedge\tau}^{T\wedge\tau} \,e^{p\rho s}|Y_s|^{p-1}\big|f(s,0,0,0)\big| ds\nn                                                                                        \\[-2pt]
 & \quad \le p\E\Biggl(\Sup_{s\in[t\wedge\t,T\wedge\t]}e^{(p-1)\rho s}|Y_{t}|^{p-1} \int_{t\wedge\tau}^{T\wedge\tau}e^{\rho s}\big|f(s,0,0,0)\big| ds \Biggr)\nn                       \\[-2pt]
 & \quad \le\frac{1}6\E\Bigl(\Sup_{s\in[t\wedge\t,T\wedge\t]}e^{p\rho s}|Y_{t}|^{p} \Bigr)+
 \xch{d'_p\E \Biggl(\int_{t\wedge\tau}^{T\wedge\tau }e^{\rho s}\big|f(s,0,0,0)\big|  ds \Biggr)^p,}{d'_p\E \Biggl(\int_{t\wedge\tau}^{T\wedge\tau }e^{\rho s}\big|f(s,0,0,0)\big|  ds \Biggr)^p}
\end{align}
and
\begin{align}
\label{young-ineq-random-time-YR}
 & p\E\int_{t\wedge\tau}^{T\wedge\tau} \,e^{p\rho s}|Y_{s-}|^{p-1} d|R|_s\nn                                                                                         \\[-2pt]
 & \quad \le\frac{1}6E\bigl(\Sup_{s\in[t\wedge\t,T\wedge\t]}e^{p\rho s}|Y_{t}|^{p} \bigr)+d''_p\E \Biggl(\int _{t\wedge\tau}^{T\wedge\tau }e^{\rho s}d|R|_s \Biggr)^p.
\end{align}
Consequently, combining \eqref{eq5.10} with \eqref
{young-ineq-random-time-Yf} and \eqref{young-ineq-random-time-YR} and
letting $T\rightarrow+\infty$, we deduce that
\begin{align}
\label{estimate-sup-Yn-random-time}
 & \E \Bigl[\sup_{t\ge0}e^{p\rho(t\wedge\t)} \big|Y^n_{t\wedge\tau }\big|^p \Bigr]\nn                                                                                \\[-2pt]
 & \quad \le C''_p \E \Biggl[e^{p\rho\tau}| \xi|^p+ \Biggl(\int_0^\t e^{\rho s}\big|f(s,0,0,0)\big| ds \Biggr)^p+ \Biggl(\int _0^{\t}e^{\rho s}d|R|_s \Biggr)^p \Biggr] .
\end{align}
Finally, going back to \eqref{ineq5.11}, we conclude in view of \eqref
{young-ineq-random-time-Yf}, \eqref{young-ineq-random-time-YR}, and
\eqref{estimate-sup-Yn-random-time} that estimate \eqref
{estimate-solution-random-terminal-time} holds for
$(Y^n,Z^n,V^n,M^n)$.

{\bf Step 2.} Let us show that $(Y^n,Z^n,V^n,M^n)$ is a Cauchy sequence.
For $m > n$, define
\[
\bar Y_t = Y^m_t - Y^n_t,
\quad\bar Z_t = Z^m_t - Z^n_t,
\quad\bar V_t = V^m_t - V^n_t,
\quad\bar M_t = M^m_t - M^n_t.
\]
For $n \leq t \leq m$, we have
\begin{align*}
\bar Y_t & = \int_{t\wedge\tau}^{m\wedge\tau} f \bigl(s,Y^m_s,Z^m_s,V^m_s \bigr) ds - \int_{t\wedge\tau}^{m\wedge\tau} \bar Z_s dW_s             \\[-2pt]
         & \quad - \int_{t\wedge\tau}^{m\wedge\tau} \int _U \bar V_s(e) \widehat\pi(du,ds) - \bar M_{m\wedge\tau} + \bar M_{t\wedge \tau}.
\end{align*}
\endgroup
Consequently, again for $n \leq t \leq m$,
{\allowdisplaybreaks
\begin{align*}
& e^{p\rho(t\wedge\tau)}| \bar{Y}_{t\wedge\tau}|^p +\alpha_p\int_{t\wedge\t}^{m\wedge\t}e^{p\rho s} | \bar{Y}_s|^{p-2} |\bar {Z}_s|^2 ds\nn                                                        \\
& \qquad +\alpha_p\int_{t\wedge\t}^{m\wedge\t}e^{p\rho s} \int_U|\bar {Y}_s|^{p-2}\big| \bar{V}_s(e)\big|^2\pi(de,ds)+\alpha_p\int _{t\wedge\t }^{m\wedge\t}e^{p\rho s}|\bar{Y}_{s-}|d[ \bar{M}]_s\nn             \\
& \quad \le  p\int_{t\wedge\t}^{m\wedge\t}e^{\rho s} \bigl[ \bar{Y}_s | \bar{Y}_s|^{p-2}f \bigl(s,Y^m_s,Z^m_s,V^m_s \bigr)-\rho|\bar{Y}_s|^p \bigr]ds\nn                                                      \\
& \qquad - p \int_{t\wedge\t}^{m\wedge\t}e^{\rho s} \bar{Y}_{s}|\bar {Y}_{s}|^{p-2} \bar{Z}_s dW_s- p \int_{t\wedge\t}^{m\wedge\t }e^{\rho s} \bar{Y}_{s-} |\bar{Y}_{s-}|^{p-2} d \bar{M}_s               \\
& \qquad - p \int_{t\wedge\t}^{m\wedge\t} \int_U e^{\rho s} \bar {Y}_{s-} |\bar{Y}_{s-}|^{p-2} \bar{V}_s(e)\,\widehat{\pi}(de,ds) \nn                                                                     \\
& \quad \le p\int_{t\wedge\t}^{m\wedge\t}e^{\rho s} \bigl[\mu| \bar {Y}_s|^{p}+L|\bar{Y}_s|^{p-1}| \bar{Z}_s|+L|\bar{Y}_s|^{p-1}\big\|\bar {V}_s(e)\big\|_{\cL^2_\lambda}-\rho|\bar{Y}_s|^p \bigr]ds                 \\[-1pt]
& \qquad + p\int_{t\wedge\t}^{m\wedge\t}e^{\rho s} \bar{Y}_s | \bar {Y}_s|^{p-2}f(s, \xi_s,\eta_s,\gamma_s) ds- p \int _{t\wedge\t }^{m\wedge\t}e^{\rho s} \bar{Y}_{s}| \bar{Y}_{s}|^{p-2} \bar{Z}_s dW_s \\[-1pt]
& \qquad - p \int_{t\wedge\t}^{m\wedge\t}e^{\rho s} \bar{Y}_{s-} |\bar {Y}_{s-}|^{p-2} d \bar{M}_s                                                                                                        \\[-1pt]
& \qquad - p \int_{t\wedge\t}^{m\wedge\t} \int_U e^{\rho s} \bar {Y}_{s-} |\bar{Y}_{s-}|^{p-2} \bar{V}_s(e)\,\widehat{\pi}(de,ds) \nn.
\end{align*}}%
We deduce by already used arguments that
\begingroup
\abovedisplayskip=3pt
\belowdisplayskip=3pt
\begin{align*}
 & \E \Biggl[\Sup_{n\le t\le m}e^{p\rho(t\wedge\tau)}| \bar {Y}_{t\wedge\tau}|^p+\int_{n\wedge\t}^{m\wedge\t}e^{p\rho s} | \bar{Y}_s|^{p} ds +\int_{n\wedge\t}^{m\wedge\t}e^{p\rho s} | \bar{Y}_s|^{p-2} |\bar{Z}_s|^2ds \nn \\[-1pt]
 & \qquad +\int_{n\wedge\t}^{m\wedge\t}e^{p\rho s} \int _U|\bar{Y}_s|^{p-2}\big\| \bar{V}_s(e)\big\|_{\cL^2_\lambda}^2 ds+\int _{n\wedge\t}^{m\wedge\t}e^{p\rho s}|\bar{Y}_{s-}|d[ \bar {M}]_s \Biggr]\nn                       \\[-1pt]
 & \quad \le C\E \Biggl(\int_{n\wedge\t}^{\t}e^{\rho s}\big|f(s, \xi_s,\eta _s,\gamma_s)\big| ds \Biggr)^p,
\end{align*}
and the last term tends to zero as $n\rightarrow\infty$.

Next, for $t \leq n$,
\begin{align*}
\bar Y_t & = \bar Y_n+\int_{t\wedge\tau}^{n\wedge\tau} \bigl(f\bigl(s,Y^m_s,Z^m_s,V^m_s \bigr)-f\bigl(s,Y^n_s,Z^n_s,V^n_s \bigr) \bigr) ds - \int_{t\wedge\tau}^{n\wedge\tau} \bar Z_s dW_s \\
         & \quad - \int_{t\wedge\tau}^{n\wedge\tau} \int _U \widehat V_s(e) \widehat\pi(de,ds) - \bar M_{n\wedge\tau} + \bar M_{t\wedge \tau}.
\end{align*}
Arguing as in Proposition~\ref{prop-uniqueness-random-time}, we get
\begin{align}
\nn \E e^{p\rho(t\wedge\tau)}| \bar{Y}_{t\wedge\tau}|^p\E\int _0^\t e^{p\rho s}|\bar{Y}_s|^p  ds & \le \E e^{p\rho(n\wedge\tau)}| \bar {Y}_{n}|^p                                                 \\[-2pt]
& \le C\E \Biggl(\int_{n\wedge\t}^{\t}e^{\rho s}\big|f(s, \xi_s,\eta _s,\gamma_s)\big| ds \Biggr)^p,\nn
\end{align}
and letting $n\rightarrow\infty$, the convergence of the sequence
$Y^n$ follows.

Next, it remains to show the convergence of the martingale part
$(Z^n,V^n,M^n)$. We follow the proof of Lemma~\ref{Lp-estimates}. We
apply It\^o's formula to $e^{2\rho s}|\bar Y_s|^2$ for $n\le t\le m$:
\begin{align*}
&\int_{t\wedge\t}^{m\wedge\t}e^{2\rho s}|\bar{Z}_s|^2 ds+\int_{t\wedge\t}^{m\wedge\t}\int_Ue^{2\rho s}\big|\bar{V}_{s}(e)\big|^2\pi (de,ds)+\int_{t\wedge\t}^{m\wedge\t}e^{2\rho s}d[\bar{M}]_s\nn\\[-1pt]
&\quad =2\int_{t\wedge\t}^{m\wedge\t}e^{2\rho s} \bigl[\bar Y_{s} \bigl(f\bigl(s,Y^m_s,Z^m_s,V^m_s\bigr)-f(s,\xi_s,\eta_s,\gamma_s) \bigr)-\rho |\bar Y_s|^2 \bigr] ds\\[-1pt]
&\qquad +2\int_{t\wedge\t}^{m\wedge\t}e^{2\rho s}\bar Y_{s}f(s,\xi _s,\eta_s,\gamma_s) ds-2\int_{0}^{\t_n}e^{2\rho s}\bar Y_{s}\bar Z_s dW_s\\[-1pt]
&\qquad -2\int_{t\wedge\t}^{m\wedge\t}\int_Ue^{2\rho s}\bar {Y}_{s-}\bar{V}_{s}(e)\widehat{\pi}(de,ds)-2\int_{t\wedge\t}^{m\wedge\t}e^{2\rho s}\bar{Y}_{s-} d\bar{M}_s .
\end{align*}
\endgroup
Mimicking the same argumentation used to obtain \eqref
{gbsde-ito2-Lp-pgreaterthan2-T} (assumptions on $f$ and BDG and Young
inequalities) leads to
\begin{align}
 &\E \Biggl[ \Biggl(\int_{n\wedge\t}^{m\wedge\t}e^{2\rho s}| \bar Z_s|^2 ds \Biggr)^\frac{p} {2}+ \Biggl( \int_{n\wedge\t}^{m\wedge\t }\int_Ue^{2\rho s}\big| \bar V_{s}(e)\big|^2 \lambda(de)ds \Biggr)^\frac {p} {2}\nn \\
 &\qquad + \Biggl(\int_{n\wedge\t}^{m\wedge\t}e^{2\rho s}d[ \bar M]_s \Biggr)^\frac{p} {2} \Biggr]\nn                                                                                                         \\
 &\quad \le C(p,T) \Biggl[\E\Sup_{t\ge n}e^{p\rho s}|\bar Y_{t}|^p+\E \Biggl(\int_{n\wedge\t}^{\t}e^{\rho s}\big|f(s, \xi_s,\eta_s,\gamma_s)\big| ds \Biggr)^p \Biggr].
\end{align}
Next, arguing similarly for the case $t\le n$, we get
\begin{align}
 & \E \Biggl[ \Biggl(\int_{0}^{n\wedge\t}e^{2\rho s}| \bar Z_s|^2 ds \Biggr)^\frac{p} {2}+ \Biggl( \int_{0}^{n\wedge\t}\int_Ue^{2\rho s}\big| \bar V_{s}(e)\big|^2 \lambda(de)ds \Biggr)^\frac{p} {2}\nn \\
 & \qquad + \Biggl(\int_{0}^{n\wedge\t}e^{2\rho s}d[ \bar M]_s \Biggr)^\frac{p} {2} \Biggr]\nn                                                                                                \\
 & \quad \le C(p,T) \bigl[\E\Sup_{t\ge n}e^{p\rho n\wedge\t}|\bar Y_{n\wedge\t}|^p \bigr].
\end{align}
Consequently, letting $n\rightarrow\infty$, we deduce from the above
that, in both cases, the sequence $(Z^n,V^n,M^n)$ is a Cauchy sequence
for the norm
\begin{equation}
\nn \E \Biggl(\int_{0}^{\t}e^{2\rho s}|\bar
Z_s|^2 ds \Biggr)^\frac {p} {2}+\E \Biggl(\int
_{0}^{\t}\int_Ue^{2\rho s}\big|
\bar V_{s}(e)\big|^2 \lambda(de)ds \Biggr)^\frac{p}
{2}
\xch{+\E \Biggl(\int_{0}^{\t}e^{2\rho s}d[\bar M]_s \Biggr)^\frac{p} {2},}{+\E \Biggl(\int_{0}^{\t}e^{2\rho s}d[\bar M]_s \Biggr)^\frac{p} {2}}
\end{equation}
and thus it converges to $(Z,V,M)$. Finally, the limit $(Y,Z,V,M)$ is a
solution of GBSDE \eqref{GBSDE-jumps-random-time} that satisfies
estimates \eqref{estimate-solution-random-terminal-time} and \eqref
{estimate-solution-random-terminal-time-2}.
\end{proof}

\bibliographystyle{vmsta-mathphys}
%

\end{document}